\documentclass[11pt]{amsart}
\usepackage{amsmath,amsfonts,amsthm,mathrsfs,amssymb,bbm,palatino,leftidx}
\usepackage[all]{xy}
\usepackage{mathtools}
\usepackage{a4wide}
\usepackage[
            breaklinks=true,
            pdftitle={},
            pdfauthor={ }]{hyperref}
\usepackage{color}
\definecolor{tocolor}{rgb}{.1,.1,.5}
\definecolor{urlcolor}{rgb}{.2,.2,.6}
\definecolor{linkcolor}{rgb}{.1,.4,.6}
\definecolor{citecolor}{rgb}{.6,.3,.1}
\hypersetup{colorlinks=true, urlcolor=urlcolor, linkcolor=linkcolor, citecolor=citecolor}

\newcommand{\C}{\mathbb{C}}

\newcommand{\Z}{\mathbb{Z}}

\mathchardef\mhyphen="2D

\renewcommand{\1}{\mathbbm{1}}

\newcommand{\arsim}{\xrightarrow{\sim}}

\newtheorem{theorem}{Theorem}[section]
\newtheorem{lem}[theorem]{Lemma}

\newtheorem{prop}[theorem]{Proposition}
\newtheorem{cor}[theorem]{Corollary}
\newtheorem{defn}[theorem]{Definition}

\numberwithin{equation}{section}

\theoremstyle{definition}

\newtheorem{rmk}[theorem]{Remark}

\newcommand{\E}{\mathcal{E}}
\newcommand{\F}{\mathcal{F}}
\newcommand{\V}{\mathcal{V}}
\newcommand{\W}{\mathcal{W}}
\newcommand{\OO}{\mathcal{O}}
\newcommand{\K}{\mathcal{K}}

\newcommand{\II}{\mathscr{I}}
\newcommand{\PP}{\mathscr{P}}
\newcommand{\QQ}{\mathscr{Q}}
\newcommand{\RR}{\mathscr{R}}
\newcommand{\G}{\mathcal{G}}
\newcommand{\M}{\mathfrak{M}}
\newcommand{\TT}{\mathscr{T}}
\newcommand{\SSS}{\mathcal{S}}
\newcommand{\Aut}{\textnormal{Aut}}
\newcommand{\Sch}{\textnormal{Sch}}
\newcommand{\Ring}{\textnormal{Ring}}
\newcommand{\Ab}{\textnormal{Ab}}
\newcommand{\Supp}{\textnormal{Supp}}
\newcommand{\Gal}{\textnormal{Gal}}
\newcommand{\PV}{\textnormal{PVect}}
\newcommand{\Vect}{\textnormal{Vect}}
\newcommand{\Hom}{\textnormal{Hom}}
\newcommand{\rk}{\textnormal{rank}}
\newcommand{\tr}{\textnormal{tr}}

\DeclareMathOperator*{\Char}{char}

\begin{document}

\title{Parabolic bundles in positive characteristic}

 \author{
  Manish Kumar
 }
 \address{
Statistics and Mathematics Unit\\
Indian Statistical Institute, \\
Bangalore, India-560059
  }
\email{manish@isibang.ac.in}

\author[S. Majumder]{Souradeep Majumder}
 \address{
Statistics and Mathematics Unit\\
Indian Statistical Institute, \\
Bangalore, India-560059
  }
\email{souradeep\_vs@isibang.ac.in}

%
%

\date{}

 \begin{abstract}
  Algebraic parabolic bundles on smooth projective curves over algebraically closed field of positive characteristic is defined. It is shown that the category of algebraic parabolic bundles is equivalent to the category of orbifold bundles defined in \cite{KP}. Tensor, dual, pullback and pushforward operations are also defined for parabolic bundles. 
 \end{abstract}

\maketitle

\section{Introduction}

Parabolic bundles on Riemann surfaces were introduced by Mehta and Shehadri in \cite{MS}. These are vector bundles $V$ on a smooth projective curve $X$ over complex numbers with a filtration of fibres of $V$ at a collection of finitely many points $S$ of $X$ and certain real numbers, called weights, attached to these filtrations. When the weights are rational it was shown that they correspond to $\Gamma$-bundles for a $\Gamma$-cover $Y\to X$ branched at $S$ with certain inertia groups depending on the weights attached to the filtration. A crucial fact that is needed for this correspondence is that over complex numbers the inertia groups are cyclic. One can fix a generator for a inertia group and obtain an automorphism of finite order of the fibre. The eigenvalues and eigenspaces of this automorphism give us the weights and the filtration respectively at that point. Similar correspondence has also been extended to higher dimensions (see \cite{Bi2}). 

In the situation when the base field has positive characteristic, one can of course define a parabolic bundle the same way as over $\C$. But this definition is not the correct one from our point of view. We would like to define a parabolic bundle in such way that we have a bijective correspondence with equivariant bundles for some suitable cover, as in the case over $\C$. The presence of wild ramification ensures that the inertia groups may no longer be cyclic and moreover they may not determine the local monodromy. So we can not hope to get such a correspondence just from the data of weights and filtrations. Indeed, as we demonstrate, we need the full data of the action of the inertia group.  

One can also talk about parabolic bundles on a curve $X$ purely in terms of $G$-bundles on $Y$ where $Y$ is a $G$-Galois cover of $X$. In \cite{KP} such bundles were called orbifold bundles and it was shown that the category of orbifold bundles do not depend on the choice of $G$-cover of $X$. Hence one could talk about the category of orbifold bundles on $X$. Though a description of these orbifold bundles as a vector bundle on $X$ together with some more data was lacking. The goal of this paper is to provide this description. In other words, we define the analogue of parabolic bundles in positive characteristic and show that they are in bijection with orbifold bundles as defined in \cite{KP}.

In section \ref{lemmas}, $G$-bundles on affine schemes are interpreted in terms of rings and $k$-algebras with $G$-action. This is used to show a local variant of the main theorem. In other words, given a faithful $G$-action on an affine scheme $Y$ and the isotropy subgroups $G_i$ of the connected components $Y_i$ of $Y$, giving a $G$-bundle on $Y$ is equivalent to giving $G_i$-bundles on $Y_i$ satisfying various compatibility conditions (Lemma \ref{constructbundle}). These lemmas along with formal patching (Theorem \ref{formalgluing}) gives the main result (Theorem \ref{genequival}) in section \ref{sec:parabolic}. This theorem says that the category of parabolic bundles on a geometric formal orbifold curve $(X,\PP)$ is equivalent to the category of $G$-bundles on $(Y,O)$ where $(Y,O)\to (X,P)$ is an \'etale $G$-Galois cover (see section \ref{sec:parabolic} for definitions).

As a consequence of the main theorem, in section \ref{application} the category of parabolic bundles on a smooth projective curve is shown to be equivalent to the category of orbifold bundles on $X$ (defined in \cite{KP}). Pullbacks and pushforward of parabolic bundles under finite morphisms are defined. The tensor product and duals of parabolic bundles are also defined. The definitions are such that the functors defining the equivalence between parabolic bundles and orbifold bundles commute with these four operations.

\subsection*{Acknowledgements}
It is a pleasure to thank A.J. Parameswaran and Vaibhav Vaish for some useful discussions. This work is partially funded by India-Israel project grant.   

\section{Preliminaries}

Let $k$ be an algebraically closed field of arbitrary characteristic. Let $X$ and $Y$ be smooth projective curves over $k$. A morphism $\pi : Y \to X$ is called a \textit{cover} if it is finite, surjective and generically separable. The \textit{automorphism group} of the cover, $\Aut(Y/X)$, is defined to be the group of automorphisms $\sigma$ of $Y$ satisfying $\pi \circ \sigma = \pi$. For a finite group $G$, $\pi$ is said to be a \textit{$G$-cover} (or \textit{$G$-Galois cover}) if we have an injective homomorphism $G \to \Aut(Y/X)$ such that $\OO_{Y}^{G} = \pi^*\OO_X$ (where the left hand side denotes the sheaf of $G$ invariants). As $X$ is a smooth curve, the last condition is equivalent to saying that $G$ acts simply transitively on a generic geometric fibre of $\pi : Y \to X$, so that $|\Aut(Y/X)| = deg(\pi)$. For a Galois cover $\Aut(Y/X)$ will also be denoted by $\Gal(Y/X)$.

Let $\pi : Y \to X$ be a $G$-Galois cover. Let $Q$ be a point of $Y$ and let $P = \pi(Q) \in X$. The \textit{decomposition group} at $Q$ is the set of $\sigma \in \Gal(Y/X)$ such that $\sigma(Q) = Q$. It is denoted by $D_Q$ and is a subgroup of $G$. The number of points in the fibre $\pi^{-1}(P)$ is $|G|/|D_Q|$. The \textit{inertia group} $I_Q$ at $Q$ is the subgroup of $D_Q$ that induces the identity automorphism on the residue field at $Q$. Since $k$ is algebraically closed, the inertia group equals the decomposition group. The cover is \textit{ramified} at $Q$ if $I_Q$ is non-trivial, and it is \textit{totally ramified} at $Q$ if $I_Q = G$. The \textit{branch locus} of $\pi$ is the set of points $P \in X$ for which there exists a ramified point $Q \in \pi^{-1}(P)$. The phrase
\textit{branched only} at $B$ means that the branch locus is contained in $B$. Clearly for two points $Q, Q^{\prime} \in \pi^{-1}(P)$ the groups $I_Q, I_{Q^{\prime}}$ are conjugates of each other. In fact, if $g.Q = Q^{\prime}$, then $I_Q = g^{-1}I_{Q^{\prime}}g$. It is well known (\cite{SerreLocal}) that inertia groups are of the form $H \rtimes \mu_r$, where $H$ is a $p$-group, $\mu_r$ is a cyclic group of order $r$ with $(p, r) = 1$ and $p>0$ is the characteristic of the field $k$. If $\Char(k)=0$ then the inertia groups are cyclic groups.

In the same situation as the previous paragraph, let $E$ be a vector bundle on $Y$. We say that $E$ is a \textit{$G$-bundle} if there is a $G$ action on $E$ which is compatible with the $G$ action on $Y$. More precisely let $\lambda : G \times Y \to Y$ be the $G$ action on $Y$ and $\E$ denotes the locally free sheaf corresponding to $E$, then $E$ is a $G$-bundle if there exists an isomorphism $\Lambda : pr_Y^*\E \arsim \lambda^*\E$ of sheaves on $G \times Y$ satisfying the following cocycle condition. For each $g \in G$, by restriction we have $\Lambda_{\{g\} \times Y} : \E \arsim \lambda_g^*\E$ where $ \lambda_g : Y \to Y$ is the isomorphism induced by $\lambda$. By identifying $\{g\} \times Y$ with $Y$ we treat this as isomorphism of sheaves on $Y$ and denote it by $\Lambda(g)$. The cocycle condition is $\Lambda(e) = \1_\E$ and for any $g, h \in G$, $\Lambda(hg) = \lambda_g^*(\Lambda(h)) \circ \Lambda(g) : \E \to \lambda_g^*\E \to \lambda_g^*\lambda_h^*\E$ (note that $\lambda_{hg}^*\E$ and $\lambda_g^*\lambda_h^*\E$ are canonically identified). As $G$ is finite, the knowledge of $\Lambda(g)$ is enough to reconstruct $\Lambda$. Note that, for any vector bundle $F$ on $X$, the pullback bundle $\pi^*F$ is naturally a $G$-bundle. We denote the category of $G$-bundles on $Y$ by $\Vect_G(Y)$.

\subsection*{Notations and conventions}

Rings are always assumed to be commutative with identity and ring homomorphisms take the identity to identity.
Points are always closed, unless specified otherwise. For any point $P \in X$ we use $\K_{X, P}$ to denote the field of fractions of the completion $\hat{\OO}_{X, P}$ of regular functions at $P$. Whenever we deal with some group of automorphisms, e.g. $\Aut_{\Ab}, \Aut_{\Ring}, \Aut_{\Sch}$, we assume that these morphisms are $k$-linear.

\section{Some generalities} \label{lemmas}

For later use we gather in this section a few results on $G$-bundles.

\begin{lem} \label{groupaction}
Let $Y = Spec(R)$ be an affine scheme and $G$ be a finite group acting on $Y$ via $\lambda : G  \to \Aut_{\Sch}(Y)$. Let $\E$ be a quasi-coherent sheaf on $Y$ with a compatible $G$ action i.e. we have isomorphisms $\Lambda(g) : \E \arsim \lambda(g)^*\E$ such that $\Lambda(e) = \1_{\E}$ and $\Lambda(hg) = \lambda(g)^*(\Lambda(h)) \circ \Lambda(g)$ for any $g, h \in G$. Let $E = \E(Y)$ be the $R$ module associated to the quasi-coherent sheaf $\E$. Then $\lambda$ corresponds to  a group homomorphism $\phi : G \to \Aut_{\Ring}(R)$ and $\Lambda$ corresponds to a group homommorphism $\Phi : G \to \Aut_{\Ab}(E)$ such that $\Phi(g)(r \cdot x) = \phi(g)(r) \cdot \Phi(g)(x)$ for any $g \in G, r \in R, x \in E$.

Moreover, let $\E^{\prime}$ be another quasi-coherent sheaf with a compatible $G$ action with $E^{\prime} = \E^{\prime}(Y)$ and the group action given by $\Phi^{\prime} : G \to \Aut_{\Ab}(E^{\prime})$ (as in the previos paragraph). Then a $G$-equivariant morphism of sheaves from $\E \to \E^{\prime}$ corresponds to $f : E \to E^{\prime}$, a morphism of $R$ modules such that the following diagram commutes 
\[
\xymatrix{
E \ar[r]^{f} \ar[d]_{\Phi(g)} & E^{\prime} \ar[d]^{\Phi^{\prime}(g)} \\
E \ar[r]^{f} & E^{\prime}
}
\]
for all $g \in G$. 
\end{lem}

\begin{proof}
For each $g \in G$ we are given $\lambda(g) : Y \arsim Y$. By the correspondence between affine schemes and rings, we have a corresponding map $\lambda(g)^{*} : R \arsim R$. Define $\phi(g) = \lambda(g^{-1})^{*}$. Clearly $\phi(g)$ is a ring automorphism and and $\phi : G \to \Aut_{\Ring}(R)$ defined by $g \mapsto \phi(g)$ gives us the required group homomorphism.

Similarly by the correspondence between quasi-coherent sheaves over affine schemes and modules over rings we have maps $\sigma(g): E \arsim (E \otimes_{R, \lambda(g)^*} R) = (E \otimes_{R, \phi(g^{-1})} R)$ induced from $\Lambda(g)$. Note that the maps $\sigma(g)$ are $R$-module maps where the $R$-module structure on $E \otimes_{R, \phi(g^{-1})} R$ is given by $r \cdot (x \otimes s) = (x \otimes rs) = (\phi(g)(r)x \otimes s)$ for any $x \in E$ and $r,s \in R$. Consider the $R$ module $E_g$ which as an abelian group is the module $E$, but the multiplication structure is given as follows : $r \cdot x := \phi(g)(r) x$ for any $r \in R, x \in E$. We can define an $R$ linear map $\beta(g) : E \otimes_{R, \phi(g^{-1})} R \rightarrow E_g$ by $x \otimes r \mapsto r \cdot x$ . This map gives us an isomorphism of $R$ modules. Define $\Phi(g) = \beta(g) \circ \sigma(g)$ which is an automorphism of $E$ as an abelian group and satisfies the required linearity condition. It remains to check that $\Phi(hg) = \Phi(h) \circ \Phi(g)$. We have
\begin{align*}
\sigma(hg) = \lambda(g)^*(\sigma(h)) \circ \sigma(g) : E \to E \otimes_{R, \phi(g^{-1})} R \to (E \otimes_{R, \phi(h^{-1})} R) \otimes_{R, \phi(g^{-1})} R
\end{align*}
where by $\lambda(g)^*(\sigma(h))$ we mean the map $\sigma(h) \otimes_{R, \phi(g^{-1})} Id_R : E \otimes_{R, \phi(g^{-1})} R \to (E \otimes_{R, \phi(h^{-1})} R) \otimes_{R, \phi(g^{-1})} R$. Note that we have canonically identified $(E \otimes_{R, \phi(h^{-1})} R) \otimes_{R, \phi(g^{-1})} R = E \otimes_{R, \phi(g^{-1}h^{-1}}) R$. Similarly 
\begin{align*}
\beta(hg) = \beta(g) \circ \lambda(g)^*(\beta(h)) : (E \otimes_{R, \phi(h^{-1})} R) \otimes_{R, \phi(g^{-1})} R \to E_h \otimes_{R, \phi(g^{-1})} R \to (E_h)_g = E_{hg}
\end{align*}
(note that $E_{hg} = (E_h)_g$). 
Now 
\begin{align*}
\Phi(hg) &= \beta(hg) \circ \sigma(hg) \\
         &= \beta(g) \circ \lambda(g)^*(\beta(h)) \circ \lambda(g)^*(\sigma(h)) \circ \sigma(g) \\
	 &= \beta(g) \circ \lambda(g)^*\Phi(h) \circ \sigma(g) \\
         &= \beta(g) \circ \lambda(g)^*\Phi(h) \circ \beta(g)^{-1} \circ \Phi(g) \\
         &= \Phi(h) \circ \Phi(g).
\end{align*}
Note that $\beta(g) \circ \lambda(g)^*\Phi(h) \circ \beta(g)^{-1} : E_g \to E \otimes_{R, \phi(g^{-1})} R \to E_h \otimes_{R, \phi(g^{-1})} R \to (E_h)_g$ is equal to the map induced by $\Phi(h)$ from $E_g \to (E_h)_g$, which we still call $\Phi(h)$. So we are done with the first part.

For dealing with morphisms, we first note that giving a $G$ equivariant morphism from $\E \to \E^{\prime}$ means that for any $g \in G$ we have a commutative diagram
\[\xymatrixcolsep{5pc}
\xymatrix{
\E \ar[r]^{\Lambda(g)} \ar[d] & \lambda(g)^{*}\E \ar[d] \\
\E^{\prime} \ar[r]^{\Lambda^{\prime}(g)} & \lambda(g)^{*}\E^{\prime}
}
\]
Correspondingly we get the following commutative diagrams:
\[\xymatrixcolsep{6pc}
\xymatrix{
E \ar[r]^{\sigma(g)} \ar[d]^{f} & E \otimes_{R, \phi(g^{-1})} R \ar[r]^{\beta(g)} \ar[d]^{f \otimes Id} & E_g \ar[d]^{f} \\
E^{\prime} \ar[r]^{\sigma^{\prime}(g)} & E^{\prime} \otimes_{R, \phi(g^{-1})} R \ar[r]^{\beta^{\prime}(g)} & E^{\prime}_g
}
\]
Clearly the result follows. 
\end{proof}

\begin{lem} \label{constructbundle}
Let $R = \prod_{i = 1}^{l} R_i$ where $R_i$'s are rings for all $1 \leq i \leq l$. Let $Y_i = Spec(R_i), Y = Spec(R)$. $G$ be a finite group acting on $Y$ via $\lambda : G  \to \Aut_{\Sch}(Y)$ and correspondingly acting on $R$ via a group homomorphism $\phi : G \to \Aut_{\Ring}(R)$. Let $G_i := Stab_G(Y_i)$. We have induced group actions $\lambda_i : G_i  \to \Aut_{\Sch}(Y_i)$ and $\phi_i : G_i \to \Aut_{\Ring}(R_i)$. Let $\E_i$ be a  coherent sheaf on $Y_i$ with compatible $G_i$ action i.e. we are given group homomorphisms $\Phi_i : G_i \to \Aut_{\Ab}(E_i)$ such that $\Phi_{i}(g)(r \cdot x) = \phi_{i}(g)(r) \cdot \Phi_{i}(g)(x)$ for any $g \in G_i, r \in R_i, x \in E_i$ where $E_i$ is the $R_i$ module corresponding to the sheaf $\E_i$. Suppose we are given the following data :
\begin{itemize}
\item[(i)] elements $\{g_{ij}\}_{1 \leq i, j \leq l}$ where $g_{ij} \in G$ for all $i, j$ such that $\lambda(g_{ji})(Y_j) = Y_i$ inducing isomorphisms of rings $\alpha_{ij} : R_i \to R_j$;
\item[(ii)] isomorphisms of abelian groups $\theta_{ij} : E_i \to E_j$ for all $i, j$;
\end{itemize}
which satisfy the following conditions :
\begin{itemize}
\item[(A)] $g_{ik} = g_{jk}g_{ij}, g_{ii} = 1$ \ \ (equivalently $g_{ik}^{-1} = g_{ij}^{-1}g_{jk}^{-1}$);
\item[(B)] $\theta_{ik} = \theta_{jk}\theta_{ij}, \theta_{ii} = Id$ \ \ (equivalently $\theta_{ik}^{-1} = \theta_{ij}^{-1}\theta_{jk}^{-1}$);
\item[(C)] $\Phi_j(g_{ij}ag_{ij}^{-1}) = \theta_{ij} \circ \Phi_i(a) \circ \theta_{ij}^{-1}$ for any $a \in G_i$;
\item[(D)] $\theta_{ij}(r \cdot x) = \alpha_{ij}(r) \cdot \theta_{ij}(x)$ for any $r \in R_i, x \in E_i$. 
\end{itemize}
Then there exists a coherent sheaf $\E$ on $Y$ with a compatible $G$ action such that $\E|Y_i$ gives back the sheaf $\E_i$ along with the $G_i$ action. 
\end{lem}

\begin{proof}
We begin by observing that : $\alpha_{ik} = \alpha_{jk}\alpha_{ij}$ (this follows from condition (A) and the fact that $\alpha_{ij}$'s are appropriate restrictions of $\phi(g_{ij})$'s) and $\phi_j(g_{ij}ag_{ij}^{-1}) = \alpha_{ij} \circ \phi_i(a) \circ \alpha_{ij}^{-1}$ for any $a \in G_i$.
Also observe the following relation between $\phi$ and $\phi_i$'s : for any $r = (r_1, \ldots, r_l) \in R, g \in G$, say $s = (s_1, \ldots, s_l) = \phi(g)(r)$. Assume $\lambda(g)(Y_i) = Y_j$ then $g_{ij}^{-1}g \in G_i$ and we can write $g = g_{ij}g^{i}$ for some $g^{i} \in G_i$. Note that $i$ and $g^{i}$ are determined by $g$ and $j$. Then $s_j = \alpha_{ij}(\phi_i(g^i)(r_i))$.

As $Y = \coprod_{i=1}^l Y_i$, we can define a coherent sheaf $\E$ on $Y$ by simply demanding that $\E|Y_i = \E_i$. The corresponding $R$ module is $E = \prod_{i=1}^l E_i$. Now we define the $G$ action on $\E$ as follows : let $v = (v_1, \ldots, v_l) \in E, v_i \in E_i$ and $w = \Phi(g)(v)$. Given $1 \leq j \leq l$ find $i$ and $g^i \in G_i$ as above. Then $w_j := \theta_{ij}(\Phi_i(g^i)(v_i))$. First we check that $\Phi$ satisfies the necessary linearity condition as described in Lemma \ref{groupaction}. Let $w^{'} = \Phi(g)(r \cdot v), r \in R$. Then by our definition 
\begin{align*}
w^{'}_j = \theta_{ij}(\Phi_i(g^i)(r_i \cdot v_i)) = \theta_{ij}(\phi_i(g^i)(r_i) \cdot \Phi_i(g^i)(v_i)) = \alpha_{ij}(\phi_i(g^i)(r_i)) \cdot w_j
\end{align*}
(here we have used the linearity condition for $\Phi_i$ and condition (D)). On the other hand, the $j$th component of $\phi(g)(r) \cdot \Phi(g)(v)$ is nothing but $s_j \cdot w_j = \alpha_{ij}(\phi_i(g^i)(r_i)) \cdot w_j$. Hence $\Phi(g)(r \cdot v) = \phi(g)(r) \cdot \Phi(g)(v)$.

Clearly $g \mapsto \Phi(g)$ gives us a function $G \to \Aut_{\Ab}(E)$. It remains to check that this defines a group homomorphism. We need to show that $\Phi(hg)(v) = \Phi(h)(\Phi(g)(v))$ for $h, g \in G, v \in E$. Let $w = \Phi(g)(v), x = \Phi(h)(w)$. Fix an index $k$, and assume $\lambda(h)(Y_j) = Y_k, \lambda(g)(Y_i) = Y_j$. As before write $g = g_{ij}g^{i}, h = g_{jk}h^{j}$ for some $g^{i} \in G_i, h^j \in G_j$. Note that $\lambda(hg)(Y_i) = Y_k$. Then using condition (A) we write :
\begin{align*}
hg = g_{jk}h^{j}g_{ij}g^{i} = g_{jk}g_{ij}(g_{ij}^{-1}h^{j}g_{ij}g^{i}) = g_{ik}(g_{ij}^{-1}h^{j}g_{ij}g^{i})
\end{align*}
By our definition $(hg)^i = g_{ij}^{-1}h^{j}g_{ij}g^{i}$. Then $(\Phi(hg)(v))_k = \theta_{ik}(\Phi_{i}((hg)^i)(v_i))$. On the other hand 
\begin{align*}
x_k = \theta_{jk}(\Phi_{j}(h^j)(w_j)) &= \theta_{jk}\Phi_{j}(h^j)\theta_{ij}(\Phi_i(g^i)(v_i)) \\
&= \theta_{jk} \theta_{ij} \Phi_{i}(g_{ij}^{-1}h^{j}g_{ij}) \theta_{ij}^{-1} \theta_{ij}(\Phi_i(g^i)(v_i)) & \text{(by condition (C))} \\
&= \theta_{ik}\Phi_{i}(g_{ij}^{-1}h^{j}g_{ij}g^{i})(v_i) & \text{(by condition (B))} \\
&= \theta_{ik}\Phi_{i}((hg)^i)(v_i).
\end{align*}
Hence clearly the necessary equality holds and we have defined a group homomorphism $G \to \Aut_{\Ab}(E)$ given by $g \mapsto \Phi(g)$. From the construction it is clear that $\E$ is the desired sheaf.  
\end{proof}

\begin{rmk}
We can construct a set of $\{g_{ij}\}$'s as described above in the following way : first choose elements $\{g_{i i+1}\}$ such that $\lambda(g_{i i+1})(Y_i) = Y_{i+1}$ for $1 \leq i \leq l-1$. Then for any $i < j$ we define $g_{ij} = g_{j-1 j} \ldots g_{i+1 i+2} g_{i i+1}$ and $g_{ji} = g_{ij}^{-1}$. Put $g_{ii} = 1$ $\forall i$ and we have the required set of elements. 
\end{rmk}

In fact we now show that the $G$-bundle constructed as in the Lemma is independent of the choice of $\{g_{ij}\}$'s. 
\begin{lem} \label{constructbundleindep}
Suppose we are in the set up of Lemma \ref{constructbundle}. Assume we are given the following data :
\begin{itemize}
\item[(i)] elements $\{g_{ij}^{\delta}\}_{1 \leq i, j \leq l}$ where $g_{ij}^{\delta} \in G$ for all $i, j$ and $\delta = 1, 2$ such that $\lambda(g_{ji}^{\delta})(Y_j) = Y_i$ inducing isomorphisms of rings $\alpha_{ij}^{\delta} : R_i \to R_j$;
\item[(ii)] isomorphisms of abelian groups $\theta_{ij}^{\delta} : E_i \to E_j$ for all $i, j$;
\end{itemize}
which satisfy the following conditions :
\begin{itemize}
\item[(A)] $g_{ik}^{\delta} = g_{jk}^{\delta}g_{ij}^{\delta}, g^{\delta}_{ii} = 1$;
\item[(B)] $\theta_{ik}^{\delta} = \theta_{jk}^{\delta}\theta_{ij}^{\delta}, \theta^{\delta}_{ii} = Id$;
\item[(C)] $\Phi_j^{\delta}(g_{ij}^{\delta}ag^{\delta}_{ij}{}^{-1}) = \theta_{ij}^{\delta} \circ \Phi_i^{\delta}(a) \circ \theta^{\delta}_{ij}{}^{-1}$ for any $a \in G_i$;
\item[(D)] $\theta_{ij}^{\delta}(r \cdot x) = \alpha_{ij}^{\delta}(r) \cdot \theta_{ij}^{\delta}(x)$ for any $r \in R_i, x \in E_i$. 
\end{itemize}
Moreover we make the assumption that $\Phi_1^1 = \Phi_1^2$. Let $\E^{\delta}$ be the coherent $G$ sheaf on $Y$ corresponding to $\{g_{ij}^{\delta}\}, \{\theta_{ij}^{\delta}\}$ as given by \ref{constructbundle}. Then we have an isomorphism of $G$ sheaves $\E^1 \cong \E^2$.
\end{lem}

\begin{proof}
In view of Lemma \ref{constructbundle}, it is enough to show that there exist a $G$-equivariant isomorphism $\tau:E\to E$ where $E=\E^1(Y)=\E^2(Y)$, and $E$ is equipped with two $G$-actions coming from the given two sets of data.

We have elements $h_i \in G_i$ such that $g_{i i+1}^2 = g_{i i+1}^1h_i$ for $1 \leq i \leq l-1$. Put $h_j^{\prime} = g_{1j}^{1}{}^{-1}h_jg_{1j}^{1}$, $h_j^{\prime} \in G_1$. Now a simple computation tells us that for any $i < j$
\begin{align*}
g_{ij}^2 = g_{1j}^1 h_{j-1}^{\prime} h_{j-2}^{\prime} \ldots h_i^{\prime} g_{1i}^{1}{}^{-1}.
\end{align*}
In particular
$g_{1j}^2 = g_{1j}^1 h_{j-1}^{\prime} h_{j-2}^{\prime} \ldots h_1^{\prime} = g_{1j}^1 f_j \ \ (\text{say}).$

Put $\Psi = \Phi_1^1 = \Phi_1^2$. Now for $b \in G_j$ 
\begin{align*}
\Phi^{2}_j(b) &= \theta^{2}_{1j} \circ \Psi(g^{2}_{1j}{}^{-1}bg^{2}_{1j}) \circ \theta^{2}_{1j}{}^{-1} \\
              &= \theta^{2}_{1j} \circ \Psi(f_j^{-1} g^{1}_{1j}{}^{-1}bg^{1}_{1j} f_j) \circ \theta^{2}_{1j}{}^{-1} \\
              &= \theta^{2}_{1j} \circ \Psi(f_j^{-1}) \circ \Psi(g^{1}_{1j}{}^{-1}bg^{1}_{1j}) \circ \Psi(f_j) \circ \theta^{2}_{1j}{}^{-1} \\
              &= (\theta^{2}_{1j} \circ \Psi(f_j^{-1}) \circ \theta^{1}_{1j}{}^{-1}) \circ \Phi^{1}_{j}(b) \circ (\theta^{1}_{1j} \circ \Psi(f_j) \circ \theta^{2}_{1j}{}^{-1}). 
\end{align*}
Setting $\tau_j = \theta^{2}_{1j} \circ \Psi(f_j^{-1}) \circ \theta^{1}_{1j}{}^{-1}:E_j\to E_j$ we obtain $\Phi^2_j(b)=\tau_j\circ\Phi^1_j(b)\circ\tau_j^{-1}$. Hence $\tau_j$ is also an isomorphism of $G_j$-modules where the $G_j$-action on the source and the target are given by $\Phi_j^1$ and $\Phi_j^2$ respectively.

Define $\tau = \prod_{i = 1}^l \tau_j : E \arsim E$. We need to check that $\Phi^{2}(g) \circ \tau = \tau \circ \Phi^{1}(g)$. For this we check the equality of $j$th component of both sides. Assume $\lambda(g)(Y_i) = Y_j$. Hence we can write $g = g^{1}_{ij}g^{i1} = g^{2}_{ij}g^{i2}$ for $g^{i\delta} \in G_i$. Now a simple computation shows us that $g^{i2}{g^{i1}}^{-1} = g^{1}_{1i} f_i f_j^{-1} {g^{1}_{1i}}^{-1}$. To complete the proof we need to check that the following diagram commutes 
\[\xymatrixrowsep{3pc}
\xymatrixcolsep{3pc}
\xymatrix{
E_i \ar[r]^{\tau_i} \ar[d]_{\theta^{1}_{ij} \circ \Phi^{1}_{i}(g^{i1})} & E_i \ar[d]^{\theta^{2}_{ij} \circ \Phi^{2}_{i}(g^{i2})} \\
E_j \ar[r]^{\tau_j} & E_j
}
\]

Equivalently $ \theta^{2}_{ij} \circ \tau_i \circ \Phi^{1}_{i}(g^{i2}) =  \tau_j \circ \theta^{1}_{ij} \circ \Phi^{1}_{i}(g^{i1})$ (by definition of $\tau_i$).

Since $\Phi^{1}_i$ is a group action and all the maps are isomorphisms it is enough to show 
$$\Phi^{1}_{i}(g^{i2}{g^{i1}}^{-1}) = \tau_i^{-1}{\theta^2_{ij}}^{-1}\tau_j\theta^1_{ij}$$
The LHS$=\Phi^{1}_{i}(g^{1}_{1i} f_i f_j^{-1} {g^{1}_{1i}}^{-1})=\theta^1_{1i}\Phi^1_1(f_if_j^{-1}){\theta^1_{1i}}^{-1}$ by condition (C). Also note that the condition (B) implies $\theta^{\delta}_{ij}=\theta^{\delta}_{1j}\theta^{\delta}_{i1}$. Hence RHS simplifies to
\begin{align*}
 \text{RHS} &= \tau_i^{-1}{\theta^2_{ij}}^{-1}\tau_j\theta^1_{ij}\\
            &= \tau_i^{-1}{\theta^2_{i1}}^{-1}{\theta^2_{1j}}^{-1}\tau_j\theta^1_{1j}\theta^1_{i1}\\
            &= \tau_i^{-1}\theta^2_{1i}\Psi(f_j^{-1})\theta^1_{i1}\\
            &= \theta_{1i}^1\Psi(f_i)\Psi(f_j^{-1}){\theta^1_{1i}}^{-1}\\
            &= \text{LHS \ \ (as $\Psi=\Phi^1_1$)}
\end{align*}

\end{proof}

\begin{lem} \label{morph}
Suppose we are in the set up of Lemma \ref{constructbundle}. In addition we are also given the data of quasi-coherent $G_i$ sheaves $\E^{\prime}_i$ over $Y_i$ with isomorphisms $\theta^{\prime}_{ij} : E^{\prime}_i \to E^{\prime}_j$ which satisfy the compatibility conditions as in Lemma \ref{constructbundle} with respect to the same $\{g_{ij}\}_{1 \leq i,j \leq l}$. Let $f_i : E_i \to E^{\prime}_i$ be $R_i$ module homomorphisms such that $\theta^{\prime}_{ij} \circ f_i = f_j \circ \theta_{ij}$  and $f_i \circ \Phi_i(g) = \Phi^{\prime}_i(g) \circ f_i$ for all $1 \leq i,j \leq l, g \in G_i$. Then there exists a $G$ equivariant morphism of quasi-coherent $G$ sheaves $f : \E \to \E^{\prime}$ such that $f|Y_i = f_i$ for all $1 \leq i \leq l$.
\end{lem}

\begin{proof}
By the proof of Lemma \ref{constructbundle} we know that $\E|Y_i = \E_i$ and $\E^{\prime}|Y_i = \E^{\prime}_i$. Hence obviously we have a morphism of quasi-coherent sheaves $f : \E \to \E^{\prime}$ such that $f|Y_i = f_i$. We just need to verify that $f$ is $G$ equvariant. More concretely, for $ v = (v_1, \ldots, v_l) \in \prod_{i=1}^l E_i, g \in G$ we show that $f(\Phi(g)(v)) = \Phi^{\prime}(g)(f(v))$. For some $j \in \{1, \ldots, l\}$ we compare the $j$-th component of both sides. As before we find $i$ and $g^{i} \in G_i$. Then by our construction $ (LHS)_j = f_j\theta_{ij}\Phi_i(g^{i})(v_i) = \theta^{\prime}_{ij}f_i\Phi_i(g^{i})(v_i) = \theta^{\prime}_{ij}\Phi^{\prime}_i(g^{i})f_i(v_i) = (RHS)_j$. Hence we are done. 
\end{proof}

We recall the following formal gluing result from \cite{Ha}, Corollary 3.1.9
\begin{theorem}
Let $V$ be a Noetherian scheme, and let $W$ be a finite set of closed points in $V$. Let $R^*$ be the ring of holomorphic functions along $W$ in $V$, let $W^* = Spec(R^*)$, let $V^0 = V - W$, and let $W^0 = W^* \times_{V} V^0$. Then the base change functor $\mathfrak{M}(V) \to \mathfrak{M}(W^*) \times_{\mathfrak{M}(W^0)} \mathfrak{M}(V^0)$ is an equivalence of categories.  
\end{theorem}
Here $\mathfrak{M}(V)$ denotes the category of coherent sheaves on $V$ and $W^*$ is nothing but the completion of $V$ along $W$.

We need the following variant of this result:
\begin{theorem} \label{formalgluing}
Let $G$ be a finite group. Let $V$ be a noetherian scheme with a $G$ action and $S$ a finite set of closed points of $V$ invariant under $G$ and on which $G$ acts transitively. Let $V^0 = V-S$, $\hat{S}$ the completion of $V$ along $S$ and $S^0 = \hat{S} \times_V V^0$. Then the base change functor is an equivalence of categories between $\mathcal{G}(V)$ and $\mathcal{G}(V^0) \times_{\mathcal{G}(S^0)} \mathcal{G}(\hat{S})$ where $\mathcal{G}(V)$ is the category of coherent $G$-sheaves of $\OO_V$ modules.  
\end{theorem}

\begin{proof}
Observe that the natural inclusion of categories $\G(V) \to \M(V)$ is faithful and the same is true for the other schemes involved. From this it easily follow that the natural morphism $\G(V^0) \times_{\G(S^0)} \G(\hat{S}) \to \M(V^0) \times_{\M(S^0)} \M(\hat{S})$ is faithful. We have a commutative diagram of categories and functors:
\[\xymatrixcolsep{5pc}
\xymatrix{
\G(V) \ar[r] \ar[d] & \G(V^0) \times_{\G(S^0)} \G(\hat{S}) \ar[d] \\
\M(V) \ar[r] & \M(V^0) \times_{\M(S^0)} \M(\hat{S})
}
\]
where the two vertical and the bottom functors are faithful, hence the top one must also be so. To show that this functor is full, assume we are given $\E, \E^{\prime} \in \G(V)$ and a morphism $(f_0, f_1) : (\E|V^{0}, \E|\hat{S}, can) \to (\E^{\prime}|V^{0}, \E^{\prime}|\hat{S}, can)$ in $\G(V^0) \times_{\G(S^0)} \G(\hat{S})$. We may consider $(f_0, f_1)$ as a morphism in $\M(V^0) \times_{\M(S^0)} \M(\hat{S})$. Hence by the above Theorem we have a morphism $f : \E \to \E^{\prime}$ in $\M(V)$ such that $f|V^0 = f_0, f|\hat{S} = f_1$. Let us denote the $G$ actions on $V, V^0, \hat{S}$ and $S^0$ by $\lambda, \lambda_0, \lambda_1$ and $\lambda^0$ respectively. By our choice $f_0$ and $f_1$ are $G$-equivariant i.e. for all $g \in G$ 
\begin{align*}
\lambda_0(g) \circ f_0 = f_0 \circ \lambda_0(g), \textnormal{and} \ \ \lambda_1(g) \circ f_1 = f_1 \circ \lambda_1(g).
\end{align*}
By making the identification $\lambda_0(g)|S^{0} = \lambda_1(g)|S^{0} = \lambda^0(g)$ we observe that the above relationships become the same when restricted to $S^{0}$. Hence again because of the previous Theorem we must have $\lambda(g) \circ f = f \circ \lambda(g)$ i.e. $f$ is $G$-equivarint. So the functor under consideration is full. 

It remains to check that this functor is also essentially surjective. Let $(\E_0, \E_1, \theta)$ be an object of $\G(V^0) \times_{\G(S^0)} \G(\hat{S})$ where $\E_0 \in \G(V^0), \E_1 \in \G(\hat{S})$ and, $\theta : \E_0|S^0 \arsim \E_1|S^0$ is an isomorphism in $\G(S^0)$. Consider $(\E_0, \E_1, \theta)$ as  an element of $\M(V^0) \times_{\M(S^0)} \M(\hat{S})$. Hence we have $\E \in \M(V)$ and isomorphisms $\E|V^0 \xrightarrow{\phi} \E_0$ in $\M(V^0)$ and $\E|\hat{S} \xrightarrow{\psi} \E_1$ in $\M(\hat{S})$ such that $\psi|S^0 = \theta \circ \phi|S^0$. For each $g \in G$, let $\Lambda_0(g) : \E_0 \to \lambda_0(g)^*\E_0$ and $\Lambda_1(g) : \E_1 \to \lambda_1(g)^*\E_1$ be the isomorphisms induced by the respective $G$ actions. As $\theta$ is an isomorphism in $\G(S^0)$ the following diagram commutes
\[\xymatrixcolsep{7pc}
\xymatrix{
\E_0|S^0 \ar[r]^{\Lambda_0(g)|S^0} \ar[d]^{\theta} & \lambda_0(g)^*\E_0|S^0 \ar[d]^{\lambda^0(g)^*\theta} \\
\E_1|S^0 \ar[r]^{\Lambda_1(g)|S^0} & \lambda_1(g)^*\E_1|S^0
}
\]
which is equivalent to saying that we have an isomorphism $(\Lambda_0(g), \Lambda_1(g)) : (\E_0, \E_1, \theta) \to (\lambda_0(g)^*\E_0, \lambda_1(g)^*\E_1, \lambda^0(g)^*\theta)$ in $\M(V^0) \times_{\M(S^0)} \M(\hat{S})$. Hence by the previous Theorem we get an isomorphism $\Lambda(g) : \E \to \lambda(g)^*\E$ (it is easy to check that $\lambda(g)^*\E$ corresponds to the triple $(\lambda_0(g)^*\E_0, \lambda_1(g)^*\E_1, \lambda^0(g)^*\theta)$). Note that we are identifying $\lambda_0(g)|S^0 = \lambda^0(g) = \lambda_1(g)|S^0$. To say that $\Lambda(g)$'s define a $G$ action on $\E$ we need to check that $\Lambda(hg) = \lambda(g)^*\Lambda(h) \circ \Lambda(g)$ for any $g, h \in G$. But we know that restricted to $\M(V^0)$ and $\M(\hat{S})$ this identity holds true (we are making the identifications $\Lambda(g)|V^0 = \Lambda_0(g)$ and $\Lambda(g)|\hat{S} = \Lambda_1(g)$). Again an application of the previous theorem tells us that the required identity holds true.

So we have constructed $\E \in \G(V)$ and from our construction it is clear that its image in $\G(V^0) \times_{\G(S^0)} \G(\hat{S})$ is isomorphic to $(\E_0, \E_1, \theta)$. Hence we are done.
\end{proof}

\section{Parabolic bundles}\label{sec:parabolic}

Let $X$ be a smooth curve over $k$. We recall the following definitions from \cite{KP}.
\begin{defn}
A quasi-branch data  on $X$ is a function $\PP$ which sends a point $x$ of $X$ to a finite Galois extension $\PP(x)$ of $\K_{X, x}$ in some fixed algebraic closure of $\K_{X, x}$. Let $\PP$ and $\PP^{\prime}$ be two quasi-branch data on $X$, we say $\PP \leq \PP^{\prime}$ if $\PP(x) \subseteq \PP^{\prime}(x)$ for all closed points $x \in X$.

The support of $\PP$, $\Supp(\PP)$ is defined to be the set of all $x \in X$ such that $\PP(x)$ is a non trivial extension of $\K_{X, x}$. A quasi-branch data $\PP$ is said to be a branch data if $\Supp(\PP)$ is a finite set. The branch data on $X$ with empty support is denoted by $O$ and is called the trivial branch data.

A smooth projective curve with a branch data is called a formal orbifold curve. 
\end{defn}

\begin{rmk}
We do not make the assumption that the underlying curve of a formal orbifold curve is always connected. 
\end{rmk}

\begin{defn} \label{parabolicdef}
Let $p$ be a point in $X$ and $\V$ be a vector bundle on $X$. A \textit{parabolic structure} on $\V$ supported on $\{p\}$ is defined by the following data: 
\begin{enumerate}
\item[(i)] a finite Galois extension $\K/\K_{X,p}$ with Galois group $I$;
\item[(ii)] a group homomorphism $\Psi : I \to \Aut_{\Ab}(\V_p \otimes_{\OO_{X, p}} R)$ where $R$ is the integral closure of $\hat{\OO}_{X, p}$ in $\K$;
\end{enumerate}
which satisfy the following conditions:
\begin{enumerate}
\item[(a)] $\Psi(g)(r \cdot x) = \psi(g)(r) \cdot \Psi(g)(x)$ for any $g \in G, r \in R$ and $x \in \V_p \otimes_{\OO_{X, p}} R$ where $\psi : I \to \Aut_{\Ring}(R)$ is the natural action;
\item[(b)] for the induced actions $\Psi^{0} = I \xrightarrow{\Psi} \Aut_{\Ab}(\V_p \otimes_{\OO_{X, p}} R) \to \Aut_{\Ab}(\V_p \otimes_{\OO_{X, p}} \K)$ and $\psi^{0} = I \xrightarrow{\psi} \Aut_{\Ring}(R) \to \Aut_{\Ring}(\K)$ we have an $I$ equivariant isomorphism $\V_p \otimes_{\OO_{X, p}} \K \xrightarrow{\mu} \V_p \otimes_{\OO_{X, p}} \K$ where the action on the left is given by $\text{Id}_{\V_p} \otimes \psi^{0}$ and the action on the right is given by $\Psi^{0}$. 
\end{enumerate}
\end{defn}

This definition easily generalizes to the situation with multiple points on $X$.
\begin{defn} 
Let $S = \{p_1, \ldots, p_N \}$ be a set of finitely many points in $X$ and $\V$ be a vector bundle on $X$. A \textit{parabolic structure} on $\V$ supported on $S$ is defined by the following data: 
\begin{enumerate}
\item[(i)] finite Galois extensions $\K_p/\K_{X, p}$ with Galois group $I_p$ for every $p \in S$;
\item[(ii)] group homomorphisms $\Psi_p : I_p \to \Aut_{\Ab}(\V_{p} \otimes_{\OO_{X, p}} R_p)$ where $R_p$ is the integral closure of $\hat{\OO}_{X, p}$ in $\K_p$ for every $p \in S$;
\end{enumerate}
which satisfy the following conditions $\forall p \in S$:
\begin{enumerate}
\item[(a)] $\Psi_p(g)(r \cdot x) = \psi_{p}(g)(r) \cdot \Psi_p(g)(x)$ for any $g \in I_p, r \in R_p$ and $x \in \V_{p} \otimes_{\OO_{X, p}} R_p$ where $\psi_p : I_p \to \Aut_{\Ring}(R_p)$ is the natural action;
\item[(b)] for the induced actions $\Psi_{p}^{0} = I_p \xrightarrow{\Psi_p} \Aut_{\Ab}(\V_{p} \otimes_{\OO_{X, p}} R_p) \to \Aut_{\Ab}(\V_{p} \otimes_{\OO_{X, p}} \K_p)$ and $\psi_{p}^{0} = I_p \xrightarrow{\psi_p} \Aut_{\Ring}(R_p) \to \Aut_{\Ring}(\K_p)$ we have $I_p$ equivariant isomorphisms $\V_{p} \otimes_{\OO_{X, p}} \K_p \xrightarrow{\mu_p} \V_{p} \otimes_{\OO_{X, p}} \K_p$ where the action on the left is given by $\text{Id}_{\V_{p}} \otimes \psi_{p}^{0}$ and the action on the right is given by $\Psi_{p}^{0}$. 
\end{enumerate}
\end{defn}

\begin{defn} \label{genparabolicdef}
Let $\PP$ be a branch data on $X$ with $\Supp(\PP) = \{p_1, \ldots, p_N \}$. By an algebraic parabolic bundle on $X$ with branch data $\PP$ we would mean a triple $(\V, \{\Psi_p\}_{p \in \Supp(\PP)}, \{\mu_p\}_{p \in \Supp(\PP)})$ where $\V$ is a vector bundle on $X$ and $(\{\PP(p)/\K_{X, p}\}_{p \in \Supp(\PP)}, \{\Psi_p\}_{p \in \Supp(\PP)}, \{\mu_p\}_{p \in \Supp(\PP)})$ is a parabolic structure on $\V$ supported on $\{p_1, \ldots, p_N \}$. 
\end{defn}

It is clear from the definition that the trivial parabolic bundle of rank $n$ on $X$ with branch data $\PP$ is nothing but the trivial rank $n$ vector bundle $\OO_X^{\oplus n}$ along with trivial action and gluing data. More precisely, for each $p \in \Supp(\PP)$ we must have $\Psi_p = Id_{\V_p} \otimes \psi_p$ and $\mu_p = Id_{\V_{p} \otimes_{\OO_{X, p}} \K_p}$. We will denote it by $\OO_X^{\oplus n}$.
\begin{defn} 
A morphism between two algebraic parabolic bundles on $X$, $(\V, \{\Psi_p\}_{p \in \Supp(\PP)}, \allowbreak \{\mu_p\}_{p \in \Supp(\PP)})$ and $(\V^{\prime}, \{\Psi^{\prime}_p\}_{p \in \Supp(\PP)}, \{\mu^{\prime}_p\}_{p \in \Supp(\PP)})$ with the same branch data $\PP$ is given by a pair $(g, \{\sigma_p\}_{p \in \Supp(\PP)})$ where $g : \V \to \V^{\prime}$ is a homomorphism of bundles and $\sigma_p : \V_{p} \otimes_{\OO_{X, p}} R_p \to \V^{\prime}_{p} \otimes_{\OO_{X, p}} R_p$ is a homomorphism of $R_p$ modules where $R_p$ is the integral closure of $\hat{\OO}_{X, p}$ in $\PP(p)$ for every $p \in \Supp(\PP)$. Also $\sigma_p$ is assumed to be $I_p$ equivariant with respect to the actions induced by $\Psi_p, \Psi^{\prime}_p$ and further it makes the following diagram commute:
\[\xymatrixcolsep{5pc}
\xymatrix{
\V_{p} \otimes_{\OO_{X, p}} \PP(p) \ar[r]^{g_{p} \otimes Id} \ar[d]^{\mu_p} & 
\V^{\prime}_{p} \otimes_{\OO_{X, p}} \PP(p) \ar[d]^{\mu^{\prime}_p} \\
\V_{p} \otimes_{\OO_{X, p}} \PP(p) \ar[r]^{\sigma^{0}_p} & \V^{\prime}_{p} \otimes_{\OO_{X, p}} \PP(p)}
\] 
where $\sigma^{0}_p$ is the map naturally induced from $\sigma_p$.
\end{defn}

We denote the category of algebraic parabolic bundles on $X$ with branch data $\PP$ by $\PV(X, \PP)$. A typical element of $\PV(X, \PP)$ will be written as $(\V, \allowbreak \{\Psi_x\}_{x \in \Supp(\PP)}, \allowbreak \{\mu_x\}_{x \in \Supp(\PP)})$, or as $(\V, \Psi, \mu)$, or just as $\V$ when the additional data is clear from the context.
\begin{rmk}
Note that $\K_{X, x} \cong k((t))$ where $t$ is a uniformizing parameter of $\hat{\OO}_{X, x}$. This would force $\Gal(\PP(x)/\K_{X, x})$ to be either a cyclic group (when characteristic of $k$ is zero) or be of the form $H \rtimes \mu_r$ where $H$ is $p$-group and $\mu_r$ is a cyclic group with $(p, r) = 1$ (when characteristic of $k$ is $p > 0$). See \cite{SerreLocal} for more details.  
\end{rmk}

\subsection*{Convention}
For the next two subsections we restrict ourselves to the case when the support of the branch data consists of only a single point i.e. $\Supp(\PP) = \{p\}$ and we denote $\K = \PP(p), I = \Gal(\PP(p)/\K_{X, p})$ and the integral closure of $\hat{\OO}_{X, p}$ in $\PP(p)$ by $R$.

\subsection{$G$-bundles to parabolic bundles} 

As before, let $\pi : Y \to X$ be a $G$-Galois cover. Let $\E$ be a $G$-bundle on $Y$. We want to construct a parabolic bundle out of $\E$ on $X$. Consider the sheaf $\pi_*\E$. The $G$ action on $\E$ induces a $G$ action on the direct image sheaf. Define $\V = (\pi_*\E)^G$, the sheaf of invariant sections. As $\pi$ is finite and flat, the sheaf $\pi_*\E$ is locally free. Clearly $(\pi_*\E)^G \subseteq \pi_*\E$ and hence locally free.

Let $B$ be the branch locus of $\pi$. Then $B$ is a finite set. For simplicity we assume $B$ contains just one point $p$. Let $\pi^{-1}(p) = \{q_1, \ldots, q_l \} = S$. Let $G_i$ be the inertia group at $q_i$ for $1 \leq i \leq l$. We know that $\K_{Y, q_i}/\K_{X, p}$ is a Galois extension with Galois group $G_i$. Let $R_i := \hat{\OO}_{Y, q_i}$, which can also be thought of as the integral closure of $\hat{\OO}_{X, p}$ in $\K_{Y, q_i} = \K_i$, and let $R = \prod_{i = 1}^l R_i$. Let $\hat{S} := Spec(R) = \coprod_{i =1}^l Spec(R_i)$ which can also be thought of as the completion of $Y$ along $S$.

Let us denote by $\lambda : G \to \Aut_{\Sch}(Y)$ the $G$ action on $Y$. Observe that $\lambda$ induces a transitive action of $G$ on $S$. This induces a transitive action of $G$ on the set of indices $\{1, \ldots, l \}$ given by : $i \mapsto g \cdot i \Leftrightarrow \lambda(g)(q_i) = q_j$. We also have an induced action of $G$ on $\hat{S}$. We call this action also as $\lambda$. Note that $G_i = Stab_G(Spec(R_i))$. Hence we have an induced action of $G_i$ on $Spec(R_i)$ denoted by $\lambda_i$.

Let $\Lambda(g) : \E \arsim \lambda(g)^{*}\E$, for any $g \in G$, denote the $G$ action on $\E$. As above $\Lambda(g)$'s induce $G$ action on $\hat{\E} := \varprojlim \E/\II_S^n\E = \E \otimes_{\OO_Y} \OO_{\hat{S}}$ compatible with the $G$ action on $\hat{S}$. Also we observe that $\hat{\E} \cong (\E \otimes_{\OO_Y} \OO_S) \otimes_{\OO_S} \OO_{\hat{S}} \cong (\prod_{i=1}^l \E_{q_i}) \otimes_{\OO_S} \OO_{\hat{S}}$. By our construction, we can naturally identify $\E_{q_i}$ with $\V_p \otimes_{\OO_{X, p}} \OO_{Y, q_i}$. Hence 
\begin{align} \label{compidentity}
\hat{\E}_{q_i} \cong \V_p \otimes_{\OO_{X, p}} \hat{\OO}_{Y, q_i}, \ \ 
\hat{\E} \cong \prod_{i=1}^l (\V_p \otimes_{\OO_{X, p}} \hat{\OO}_{Y, q_i}).
\end{align}
Let $\hat{\E}_i := \hat{\E} \otimes_{\OO_{\hat{S}}} R_i \cong \V_p \otimes_{\OO_{X, p}} \hat{\OO}_{Y, q_i}$. Then we have an induced $G_i$ action on the bundle $\hat{\E}_i$, denoted by $\Lambda_i$, which is compatible with $\lambda_i$.

Now using Lemma \ref{groupaction} we would restate all the data obtained in terms of rings and modules. $\lambda : G \to \Aut_{\Sch}(\hat{S})$ corresponds to a group homomorphism $\phi : G \to \Aut_{\Ring}(R)$. For any $g \in G$ and $1 \leq i \leq l$ we have we have isomorphisms $\alpha(g) : R_i \arsim R_{g \cdot i}$ induced by $\phi(g)$. $\lambda_i : G_i \to \Aut_{\Sch}(Spec(R_i))$ corresponds to a group homomorphism $\phi_i : G_i \to \Aut_{\Ring}(R_i)$. From our definitions it is clear that $\phi_{g \cdot i}(gag^{-1}) = \alpha(g) \circ \phi_i(a) \circ \alpha(g)^{-1}$ for any $a \in G_i$.

Now we take note of the fact that $R_i$ is a $DVR$ for $1 \leq i \leq l$. Hence $\hat{\E}_i(Spec(R_i)) = \V_p \otimes_{\OO_{X, p}} R_i$ and $\hat{\E}_i(Spec(\K_i)) = \V_p \otimes_{\OO_{X, p}} \K_i$. Further $\hat{\E}(\hat{S}) = \prod_{i=1}^l \hat{\E}_i(Spec(R_i)) \cong \prod_{i=1}^l \V_p \otimes_{\OO_{X, p}} R_i$. The $G$ action on $\hat{\E}$ corresponds to a group homomorphism $\Phi : G \to \Aut_{\Ab}(\prod_{i=1}^l \V_p \otimes_{\OO_{X, p}} R_i)$, such that $\Phi(g)(r \cdot x) = \phi(g)(r) \cdot \Phi(g)(x)$ for any $g \in G, r \in R, x \in \prod_{i=1}^l \V_p \otimes_{\OO_{X, p}} R_i$. As before $\Phi(g)$ induces isomorphisms $\theta(g) : \V_p \otimes_{\OO_{X, p}} R_i \arsim \V_p \otimes_{\OO_{X, p}} R_{g \cdot i}$. Similarly the $G_i$ action on $\hat{\E}_i$ corresponds to a group homomorphism $\Phi_i : G \to \Aut_{\Ab}(\V_p \otimes_{\OO_{X, p}} R_i)$ such that $\Phi_i(g)(r \cdot x) = \phi_i(g)(r) \cdot \Phi_i(g)(x)$ for any $g \in G_i, r \in R_i, x \in \V_p \otimes_{\OO_{X, p}} R_i$. As before, we have $\Phi_{g \cdot i}(gag^{-1}) = \theta(g) \circ \Phi_i(a) \circ \theta(g)^{-1}$ for any $a \in G_i, g \in G$. We also have $\theta(g)(r \cdot x) = \alpha(g)(r) \cdot \theta(g)(x)$ for any $g \in G, r \in R_i, x \in \V_p \otimes_{\OO_{X, p}} R_i$ (this is nothing but the linearity condition on $\Phi(g)$ restated).

Let $S^0 := \hat{S} \times_{Y} Y^0$. Note that $S^0 \cong \coprod_{i=1}^l Spec(\K_i)$ and we have a $G$ action on $S^0$ induced by $\phi$. In our situation we have $\E|Y^0 \in \mathcal{G}(Y^0)$ and $\hat{\E} \in \mathcal{G}(\hat{S})$. As the map $\pi : Y^0 \to X^0 = X-B$ is unramified Galois, we know that $\E|Y^0 \cong \pi^*(\V|X^0)$ and the $G$ action on $\pi^*(\V|X^0)$ is the natural action induced on pullback sheaves. After base change to $S^0$ we have $\E|Y^0 \otimes_{\OO_{Y^0}} \OO_{S^0} \in \mathcal{G}(S^0)$ with the induced $G$ action. Note that $\E|Y^0 \otimes_{\OO_{Y^0}} \OO_{S^0} \cong \prod_{i=1}^l (\V_p \otimes_{\OO_{X, p}} \K_i)$ and the $G$ action on $\prod_{i=1}^l (\V_p \otimes_{\OO_{X, p}} \K_i)$ is defined by the composite : $G \xrightarrow{\phi} \Aut_{\Ring}(R) \to \Aut_{\Ring}(\prod_{i=1}^l \K_i)$, which we denote by $\phi^0$.

On the other hand we have $\hat{\E}$ with $G$ action given by $\Lambda(g)$'s. After base change to $S^0$ we have $\hat{\E} \otimes_{\OO_{\hat{S}}} \OO_{S^0} \in \mathcal{G}(S^0)$ with the induced $G$ action. Note that $\hat{\E} \otimes_{\OO_{\hat{S}}} \OO_{S^0} \cong \prod_{i=1}^l (\V_p \otimes_{\OO_{X, p}} \K_i)$ and the $G$ action on $\prod_{i=1}^l (\V_p \otimes_{\OO_{X, p}} \K_i)$ is defined by $G \xrightarrow{\Phi} \Aut_{\Ab}(\prod_{i=1}^l \V_p \otimes_{\OO_{X, p}} R_i) \to \Aut_{\Ab}(\prod_{i=1}^l \V_p \otimes_{\OO_{X, p}} \K_i)$, which we denote by $\Phi^0$. Now we have a canonical isomorphism of $G$-sheaves $\E|Y^0 \otimes_{\OO_{Y^0}} \OO_{S^0} \arsim \hat{\E} \otimes_{\OO_{\hat{S}}} \OO_{S^0}$. In other words an isomorphism of $G$-sheaves $\tau : \prod_{i=1}^l (\V_p \otimes_{\OO_{X, p}} \K_i) \arsim \prod_{i=1}^l (\V_p \otimes_{\OO_{X, p}} \K_i)$ where the $G$ action on the source is given by $\phi^0$ and on the target it is given by $\Phi^0$. As before we have induced isomorphism of $G_i$-sheaves $\tau_i : \V_p \otimes_{\OO_{X, p}} \K_i \arsim \V_p \otimes_{\OO_{X, p}} \K_i$ which satisfy $\tau_{g \cdot i} = \theta(g) \circ \tau_i \circ \theta(g)^{-1}$ for any $g \in G$.

All of the above can be summarized as :
\begin{prop} \label{orbitopara}
Let $\pi : Y \to X$ be a $G$-Galois cover with branch locus $\{p\}$. Let $\E$ be a $G$-bundle on $Y$. Then we can construct an algebraic parabolic bundle $\V$ on $X$ with branch data $\PP$ such that $\Supp(\PP) = \{p\}$.
\end{prop}

\begin{proof}
We define $\V = (\pi_{*}\E)^{G}, \K/\K_{X, p} := \K_{1}/\K_{X, p}, \Psi := \Phi_{1}, \mu := \tau_{1}$. From the above discussion it is clear that this defines an algebraic parabolic bundle on $X$ with branch data supported on $\{p\}$.
\end{proof}   

Now given a morphism of two $G$-bundles on $Y$, $ f : \E \to \E^{\prime}$, we have an induced homomorphism of bundles $g = (\pi_{*}f)^G : \V \to \V^{\prime}$ where $\V^{\prime} = (\pi_{*}\E^{\prime})^G$. Similarly we have induced morphisms of $G_i$-bundles $\hat{f}_i : \hat{\E}_i \to \hat{\E^{\prime}}_i$. Note that $f|Y^{0} = \pi^{*}(g|X^{0})$. 
\begin{prop}
Let $\pi : Y \to X$ be a $G$-Galois cover with branch locus $\{p\}$. Let $\E, \E^{\prime}$ be $G$-bundles on $Y$ and let $f : \E \to \E^{\prime}$ be a morphism of $G$-bundles. Let $\V, \V^{\prime}$ be the algebraic parabolic bundles constructed on $X$ with branch data $\PP$ according to Proposition \ref{orbitopara}. Then there is a morphism $(g, \sigma) : \V \to \V^{\prime}$ of algebraic parabolic bundles with branch data $\PP$.
\end{prop} 

\begin{proof}
Put $\sigma = \hat{f}_1$ and the result follows.
\end{proof}

\subsection{Parabolic bundles to $G$-bundles}
 
Let $(\V, \Psi, \mu)$ be an algebraic parabolic bundle on $X$ with branch data $\PP$. Suppose we are given a $G$-Galois cover $\pi : Y \to X$ such that
\begin{enumerate}
\item[(i)] $\pi$ is branched only at $p$ with $\pi^{-1}(p) = \{y_1, \ldots, y_l\} = S$;
\item[(ii)] let $\K_i$ be the quotient field of $R_i := \hat{\OO}_{Y, y_i}$, then the extension $\K_i/\K_{X, p}$ is isomorphic to the extension $\K/\K_{X, p}$ for $1 \leq i \leq l$.
\end{enumerate}
In particular we have induced isomorphisms of Galois groups $G_i \cong I$ for $1 \leq i \leq l$. We would like to construct a $G$-bundle on $Y$ from this data. 

Without loss of generality we may assume that $\K_1 = \K$ and consequently $R_1 = R, G_1 = I$. As discussed before, the data of a $G$-Galois cover gives us group homomorphisms $\phi : G \to \Aut_{\Ring}(\prod_{i=1}^l R_i)$ and $\phi_i : G_i \to \Aut_{\Ring}(R_i)$ for $1 \leq i \leq l$. Note that by our assumption $\phi_1 = \psi : I \to \Aut_{\Ring}(R)$. 

Let us fix elements $g_{ij} \in G$ such that $g_{ij} \cdot y_i = y_j$ for $1 \leq i, j \leq l$ and $g_{ik} = g_{jk}g_{ij}, g_{ii} = 1$. Then we have induced isomorphisms $\alpha_{ij} : R_i \to R_j$ which satisfy $\phi_j(g_{ij}ag_{ij}^{-1}) = \alpha_{ij} \circ \phi_{i}(a) \circ \alpha_{ij}^{-1}$ for $a \in G_i$ and $\alpha_{ik} = \alpha_{jk} \circ \alpha_{ij}, \alpha_{ii} = Id$. Define $\theta_{ij} : \V_p \otimes_{\OO_{X, p}} R_i \to \V_p \otimes_{\OO_{X, p}} R_j$ by $x \otimes r \mapsto x \otimes \alpha_{ij}(r)$. Set $\Phi_1 = \Psi : G_1 \to \Aut_{\Ab}(\V_p \otimes_{\OO_{X, p}} R_1)$. Define $\Phi_j : G_j \to \Aut_{\Ab}(\V_p \otimes_{\OO_{X, p}} R_j)$ by $\Phi_j(b) = \theta_{1j} \circ \Phi_1(g_{1j}^{-1}bg_{1j}) \circ \theta_{1j}^{-1}$. Then the following holds
\begin{itemize}
\item $\theta_{ik} = \theta_{jk} \circ \theta_{ij}$;
\item $\theta_{ij}(r \cdot x) = \alpha_{ij}(r) \cdot \theta_{ij}(x)$ for $r \in R_i, x \in \V_p \otimes_{\OO_{X, p}} R_i$;
\item $\Phi_j(g_{ij}ag_{ij}^{-1}) = \theta_{ij} \circ \Phi_{i}(a) \circ \theta_{ij}^{-1}$ for $a \in G_i$;
\item $\Phi_{i}(g)(r \cdot x) = \phi_{i}(g)(r) \cdot \Phi_{i}(g)(x)$ for any $g \in G_i, r \in R_i, x \in \V_p \otimes_{\OO_{X, p}} R_i$.
\end{itemize}
Hence, by Lemma \ref{constructbundle}, we have a $G$-bundle $\E^1$ on $\hat{S}$ corresponding to a group homomorphism $G \xrightarrow{\Phi} \Aut_{\Ab}(\prod_{i=1}^l \V_p \otimes_{\OO_{X, p}} R_i)$. Note that by Lemma \ref{constructbundleindep} the $G$-bundle $\E^1$ is independent of our choice of $\{g_{ij}\}$'s.

Let $\E^0 := \pi^{*}(\V|X^{0})$ which is a $G$-bundle on $Y^{0}$ with a natural $G$ action. As before $\E^{0} \otimes_{\OO_{Y^0}} \OO_{S^0} \cong \prod_{i=1}^l (\V_p \otimes_{\OO_{X, p}} \K_i)$ and the $G$ action on $\prod_{i=1}^l (\V_p \otimes_{\OO_{X, p}} \K_i)$ is defined by the composite : $G \xrightarrow{\phi} \Aut_{\Ring}(R) \to \Aut_{\Ring}(\prod_{i=1}^l \K_i)$, which we denote by $\phi^0$. Similarly we can check that $\E^1 \otimes_{\OO_{\hat{S}}} \OO_{S^0} \cong \prod_{i=1}^l (\V_p \otimes_{\OO_{X, p}} \K_i)$ and the $G$ action on $\prod_{i=1}^l (\V_p \otimes_{\OO_{X, p}} \K_i)$ is defined by $G \xrightarrow{\Phi} \Aut_{\Ab}(\prod_{i=1}^l \V_p \otimes_{\OO_{X, p}} R_i) \to \Aut_{\Ab}(\prod_{i=1}^l \V_p \otimes_{\OO_{X, p}} \K_i)$, which we denote by $\Phi^0$.

Define a map $\tau : \prod_{i=1}^l \V_p \otimes_{\OO_{X, p}} \K_i \to \prod_{i=1}^l \V_p \otimes_{\OO_{X, p}} \K_i$ by $\tau(x_1, \ldots, x_l) = (y_1, \ldots, y_l)$ where $y_j := \theta_{1j} \circ \mu \circ \theta_{1j}^{-1}(x_j)$. Note that we are denoting the map induced by $\theta_{ij}$ from $\V_p \otimes_{\OO_{X, p}} \K_i \to \V_p \otimes_{\OO_{X, p}} \K_i$ by $\theta_{ij}$. Then it follows from the definitions that $\tau$ is $G$- equivariant where $G$ acts on the source by $\phi^0$ and on the target by $\Phi^0$ and hence an isomorphism of $G$-bundles. By Theorem \ref{formalgluing} we get a $G$ bundle $\E$ on $Y$.
 
\begin{prop} \label{paratoorbi}
Let $(\V, \Psi, mu)$ be an algebraic parabolic bundle on $X$ with branch data $\PP$ as defined in Definition \ref{genparabolicdef}. Suppose we are given a $G$-Galois cover $\pi : Y \to X$ such that
\begin{enumerate}
\item[(i)] $\pi$ is branched only at $p$ with $\pi^{-1}(p) = \{y_1, \ldots, y_l\}$;
\item[(ii)] let $\K_i$ be the quotient field of $R_i := \hat{\OO}_{Y, y_i}$, then the extension $\K_i/\K_{X, p}$ is isomorphic to the extension $\K/\K_{X, p}$ for $1 \leq i \leq l$.
\end{enumerate}
Then we can construct a $G$ bundle $\E$ on $Y$ which gives back the original algebraic parabolic bundle when we apply the construction in Proposition \ref{orbitopara}. 
\end{prop}

\begin{proof}
It is immediate from the construction above and the construction in \ref{orbitopara}.
\end{proof}

Let $(\V^{\prime}, \Psi^{\prime}, \mu^{\prime})$ be another algebraic parabolic bundle on $X$ with the same branch data $\PP$. By the above Proposition we get a $G$ bundle $\E^{\prime}$ on $Y$. Let $(g, \sigma)$ be a morphism of parabolic bundles from $(\V, \Psi, \mu) \to (\V^{\prime}, \Psi^{\prime}, \mu^{\prime})$. Define $f_1 = \sigma : (\V_p \otimes_{\OO_{X, p}} R_1) \to (\V^{\prime}_p \otimes_{\OO_{X, p}} R_1)$ and $f_j = \theta^{\prime}_{1j} \circ f_1 \circ \theta_{1j}^{-1} : (\V_p \otimes_{\OO_{X, p}} R_j) \to (\V_p \otimes_{\OO_{X, p}} R_j)$, for $2 \leq j \leq l$. Then it is easy to see that the condition of Lemma \ref{morph} are satisfied and hence we have a morphism of $G$-bundles $\hat{\E} \to \hat{\E^{\prime}}$. We also have the morphism $\pi^{*}(g|X^{0}) : \E|Y^{0} \to \E^{\prime}|Y^{0}$ of trivial $G$-bundles. The compatibility condition between $g$ and $\sigma$ allows us to use Theorem \ref{formalgluing} and we get a morphism of $G$-bundles $f : \E \to \E^{\prime}$.

Finally we have
\begin{theorem} \label{equival}
Let $X, Y$ be smooth projective algebraic curves over an algebraically closed field $k$. Let $\PP$ be a branch data on $X$. Let $\pi : Y \to X$ be a morphism which makes $Y$ into a $G$-Galois cover of $X$ such that
\begin{enumerate}
\item[(i)] $\pi$ is branched only at $p$ with $\pi^{-1}(p) = \{y_1, \ldots, y_l\}$;
\item[(ii)] let $\K_i$ be the quotient field of $R_i := \hat{\OO}_{Y, y_i}$, then the extension $\K_i/\K_{X, p}$ is isomorphic to the extension $\K/\K_{X, p}$ for $1 \leq i \leq l$.
\end{enumerate}
Then the category $\PV(X, \PP)$ is equivalent to the category $\Vect_G(Y)$. 
\end{theorem}

\begin{proof}
Starting from a $G$-bundle $\E$ on $Y$ we construct an algebraic parabolic bundle $(\V, \Psi, \mu)$ on $X$ with branch data $\PP$ as in Proposition \ref{orbitopara}. Let $\tilde{\E}$ be the $G$-bundle on $Y$, constructed as in Proposition \ref{paratoorbi}, from $(\V, \Psi, \mu)$. We need to show that $\E \cong \tilde{\E}$ as $G$-bundles. 

Let $S = \pi^{-1}(p)$. First we construct an isomorphism $\E \otimes_{\OO_{Y}} \OO_{\hat{S}} \arsim \tilde{\E} \otimes_{\OO_{Y}} \OO_{\hat{S}}$. Note that as $\prod_{i=1}^lR_i$ modules, both $\E \otimes_{\OO_{Y}} \OO_{\hat{S}}$ and $\tilde{\E} \otimes_{\OO_{Y}} \OO_{\hat{S}}$ are the same module viz. $\prod_{i=1}^l (\V_p \otimes_{\OO_{X, p}} R_i)$. Define $\rho_1 : \V_p \otimes_{\OO_{X, p}} R_1 \to \V_p \otimes_{\OO_{X, p}} R_1$ to be the identity map. By our construction in Proposition \ref{paratoorbi} $\Phi_1 = \Psi$, clearly $\rho_1$ is an isomorphism of $G_1$ bundles. Define $\rho_i = (Id_{\V_{p}} \otimes \alpha_{1i}) \circ \theta^{-1}_{1i}$ from $\V_p \otimes_{\OO_{X, p}} R_i \to \V_p \otimes_{\OO_{X, p}} R_i$ for $i > 1$. Then for $a \in G_i$
\begin{align*}
\rho_i \circ \Phi_i(a) & = (Id_{\V_{p}} \otimes \alpha_{1i}) \circ \theta^{-1}_{1i} \circ \theta_{1i} \circ \Phi_1(g_{1i}^{-1}ag_{1i}) \circ \theta^{-1}_{1i} \\
& = (Id_{\V_{p}} \otimes \alpha_{1i}) \circ \Psi(g_{1i}^{-1}ag_{1i}) \circ \theta^{-1}_{1i} \\
& = (Id_{\V_{p}} \otimes \alpha_{1i}) \circ \Psi(g_{1i}^{-1}ag_{1i}) \circ (Id_{\V_{p}} \otimes \alpha_{1i})^{-1} \circ (Id_{\V_{p}} \otimes \alpha_{1i}) \circ \theta^{-1}_{1i} \\
& = \tilde{\Phi}_i(a) \circ \rho_i 
\end{align*}
where $\tilde{\Phi}_i$ corresponds to the $G$ action on $\tilde{\E}$ restricted to $G_i$ as constructed in Proposition \ref{paratoorbi}. Hence $\rho_i$ is $G_i$-equivariant. Define $\rho = \prod_{i=1}^l \rho_i : \prod_{i=1}^l (\V_p \otimes_{\OO_{X, p}} R_i) \to \prod_{i=1}^l (\V_p \otimes_{\OO_{X, p}} R_i)$. It can be easily checked that $\rho$ is an isomorphism of $G$-bundles. After changing the base to $S^0$ we would denote this isomorphism by $\rho^0$. Note that as per our construction, on $Y^{0} = Y - S$ we have canonical isomorphism of $G$ bundles $\E|Y^{0} \cong \pi^{*}(\V|X^{0}) = \tilde{\E}|Y^{0}$ where $X^{0} = X - \{p\}$. Now one can easily see that after base change to $S^0$, this isomorphism is nothing but $\rho^0$ (recall how $G$ acts on $\pi^{*}(\V|X^{0})$) .

Now the gluing data is given by the isomorphisms $\tau = \prod_{i=1}^l \tau_i, \tilde{\tau} = \prod_{i=1}^l \tilde{\tau}_i$ corresponding to $\E, \tilde{\E}$ respectively. By our construction $\tau_i = \theta_{1i} \circ \tau_1 \circ \theta^{-1}_{1i}$, $\tilde{\tau}_i = (Id_{\V_{p}} \otimes \alpha_{1i}) \circ \mu \circ (Id_{\V_{p}} \otimes \alpha_{1i})^{-1}$ and $\tau_1 = \mu$. It is now trivial to check that $\tilde{\tau}_i \circ \rho_i^0 = \rho_i^0 \circ \tau_i$ and hence $\tilde{\tau} \circ \rho^0 = \rho^0 \circ \tau$. Thus we have the necessary isomorphism $\E \arsim \tilde{\E}$.

It is obvious that if we start with an algebraic parabolic bundle on $X$ with branch data $\PP$ then construct the associated $G$ bundle on $Y$ (as in Proposition \ref{paratoorbi}) and from that construct the associated algebraic parabolic bundle on $X$ (as in Proposition \ref{orbitopara}) we get the bundle we started with.
\end{proof}

\subsection{Generalization}

We now generalize to the situation when the branch locus  and $\Supp(\PP)$ contains more than one point.

\begin{theorem} \label{genequival}
Let $X, Y$ be smooth projective algebraic curves over an algebraically closed field $k$. Let $S = \{p_1, \ldots, p_N \}$ be a set of finitely many points in $X$ and  let $\PP$ be a branch data on $X$ with $\Supp(\PP) = S$. Let $\pi : Y \to X$ be a morphism which makes $Y$ into a $G$-Galois cover of $X$ such that
\begin{enumerate}
\item[(i)] $\pi$ is ramified precisely at $S$ with $\pi^{-1}(p_i) = \{y_{i1}, \ldots, y_{i{l_i}} \}$;
\item[(ii)] let $\K_{ij}$ be the quotient field of $R_{ij} := \hat{\OO}_{Y, y_{ij}}$, then the extension $\K_{ij}/\K_{X, p_i}$ is isomorphic to the extension $\PP(p_i)/\K_{X, p_i}$ for $1 \leq j \leq l_i$ and $1 \leq i \leq N$.
\end{enumerate}
Then the category $\PV(X, \PP)$ is equivalent to the category $\Vect_G(Y)$. 
\end{theorem}

\begin{proof} 
Let $X_0 = X - S, X_1 = X_0 \cup \{p_1\}, \ldots, X = X_N = X_{N-1} \cup \{p_N\}$. Given a parabolic bundle on $X$ with branch data $\PP$, by repeated use of Theorem \ref{equival} we successively construct $G$ bundles on $\pi^{-1}(X_1), \ldots, \pi^{-1}(X_{N-1})$ and finally on $\pi^{-1}(X_N) = Y$. The construction of a parabolic bundle out of a $G$ bundle on $Y$ is exactly same as before. Thus we have our equivalence.
\end{proof}

\begin{rmk}
One could modify the proof of Theorem \ref{equival} by working simultaneously with multiple branch points and obtain Theorem \ref{genequival} directly. This has been avoided just to simplify the notation.
\end{rmk}

\begin{rmk}
It is clear from our construction that under this equivalence the trivial $G$-bundle corresponds to the trivial parabolic bundle with branch data $\PP$.
\end{rmk}  

We recall the following definitions from \cite{KP}
\begin{defn}
Let $(X, \PP)$ and $(Y, \QQ)$ be formal orbifold curves. A morphism of formal orbifolds $f : (Y, \QQ) \to (X, \PP)$ is a finite cover $f : Y \to X$ such that for all $y \in Y$ the extension $\QQ(y)/\K_{X, f(y)}$ contains the extension $\PP(f(y))/\K_{X, f(y)}$.

A morphism of formal orbifolds $f : (Y, \QQ) \to (X, \PP)$ is called e\'tale at $y$ if $\QQ(y) = \PP(f(y))$ and is called e\'tale if it is e\'tale for all points $y \in Y$.

We say that a formal orbifold $(X, \PP)$ is geometric if there exists a connected e\'tale cover $(Y, O) \to (X, \PP)$ where $O$ is the trivial branch data on $Y$. In this situation $\PP$ is called a geometric branch data on $X$.

Let $f : (Y, \QQ) \to (X, \PP)$ be a morphism of formal orbifolds. It is called a $G$-Galois cover for a finite group $G$ if $f : Y \to X$ is a $G$-Galois cover, $\QQ(y)/\PP(f(y))$ is a Galois extension for all $y \in Y$ and for all $g \in G, y \in Y$ , the extension $\QQ(y)/\K_{X, f(y)}$ is isomorphic to $\QQ(gy)/\K_{X, f(y)}$. 
\end{defn}

\begin{rmk}
Let $\PP, \PP^{\prime}$ be two branch data on $X$. Then $Id_X$ defines a morphism of fromal orbifolds $(X, \PP^{\prime}) \to (X, \PP)$ iff $\PP^{\prime} \geq \PP$ and the morphism is e\'tale iff $\PP^{\prime} = \PP$.

Given any branch data $\PP$ on $X$ we can always find a branch data $\QQ$ such that $\PP \leq \QQ$ and $\QQ$ is a geometric branch data on $X$.
\end{rmk}

\subsection*{Notation}

In Theorem \ref{genequival} we assume that $\PP$ is a geometric branch data and that there is a $G$-Galois \'etale cover $(Y, O) \to (X, \PP)$. The equivalence of the categories $\PV(X, \PP)$ and $\Vect_G(Y)$ has been shown only under this assumption. In this situation let us denote the functor we have constructed from $\PV(X, \PP) \to \Vect_G(Y)$ by $\TT_{(X, \PP)}^Y$ and the one from $\Vect_G(Y) \to \PV(X, \PP)$ by $\SSS_{(X, \PP)}^Y$.

\begin{rmk}
Suppose we have a commutative diagram of formal orbifolds
\[
\xymatrix{
(Z, O) \ar[r]^g \ar[rd]_{f \circ g} & (Y, O) \ar[d]^f \\
& (X, \PP)
}
\]
where both $(Z, O) \xrightarrow{f \circ g} (X, \PP)$ and $(Y, O) \xrightarrow{f} (X, \PP)$ are Galois \'etale with Galois groups $\Gamma$ and $G$ respectively. Then obviously $(Z, O) \xrightarrow{g} (Y, O)$ is also Galois \'etale with Galois group $H$ such that $H$ is a normal subgroup of $\Gamma$ and $\Gamma/H \cong G$. 

Now we know the following fact : $g^*$ defines an equivalence of categories $\Vect_G(Y) \arsim \Vect_{\Gamma}(Z)$ (see \cite{KP}). An inverse to $g^*$ is given by the functor $g_*^H$ which takes $\W \mapsto (g_*\W)^H$ for $\W \in \Vect_{\Gamma}(Z)$. One can check that $g^* = \TT_{(Y, O)}^Z$, modulo the identification $\PV(Y, O) = \Vect(Y)$ and after restricting both the functors to $\Vect_G(Y)$. Similarly we also have $g_*^H = \SSS_{(Y, O)}^Z$ applied to $\Vect_{\Gamma}(Z)$. Moreover by our construction the following holds in this situation
\begin{align*}
\TT_{(X, \PP)}^Z = g^* \circ \TT_{(X, \PP)}^Y, g_*^H \circ \TT_{(X, \PP)}^Z = \TT_{(X, \PP)}^Y, \SSS_{(X, \PP)}^Z \circ g^* = \SSS_{(X, \PP)}^Y 
\end{align*}
or, equivalently
\begin{align*}
\TT_{(X, \PP)}^Z = \TT_{(Y, O)}^Z \circ \TT_{(X, \PP)}^Y, \SSS_{(Y, O)}^Z \circ \TT_{(X, \PP)}^Z = \TT_{(X, \PP)}^Y, \SSS_{(X, \PP)}^Z \circ \TT_{(Y, O)}^Z = \SSS_{(X, \PP)}^Y 
\end{align*}  
(here we have replaced canonical isomorphisms by equality).
\end{rmk}

\begin{cor}
Let $(X, \PP)$ be a geometric formal orbifold. Then the category of algebraic parabolic bundles on $X$ with branch data $\PP$ is equivalent to the category of "vector bundles" on $(X, \PP)$ as defined in \cite{KP}, Definition 3.5.
\end{cor}

\begin{proof}
As per the definitions in \cite{KP}, a geometric formal orbifold $(X, \PP)$ comes with a $G$-Galois cover $Y \to X$ which satisfies the conditions stated in the Theorem above. Also, a "vector bundle" on $(X, \PP)$ was defined as a $G$-bundle on $Y$. Hence from Theorem \ref{genequival} the statement follows. 
\end{proof}

\section{Orbifold bundles vs parabolic bundles} \label{application}
Now we proceed to define the category of parabolic bundles on a smooth projective curve and show that it is equivalent to the category of "orbifold bundles" on $X$ as defined in \cite{KP}.
\begin{defn} \label{pequiv}
Let $\PP, \PP^{\prime}$ be two branch data on $X$ such that $\PP \leq \PP^{\prime}$. Let $(\V, \{\Psi_x\}_{x \in \Supp(\PP)}, \allowbreak \{\mu_x\}_{x \in \Supp(\PP)})$ and $(\V^{\prime}, \{\Psi^{\prime}_x\}_{x \in \Supp(\PP^{\prime})}, \{\mu^{\prime}_x\}_{x \in \Supp(\PP^{\prime})})$ be two algebraic parabolic bundles on $X$ with branch data $\PP$ and $\PP^{\prime}$ respectively. We say that these two algebraic parabolic bundles are equivalent if (i) $\V \cong \V^{\prime}$, (ii) $\forall x \in \Supp(\PP^{\prime}), g \in \Gal(\PP^{\prime}(x)/\K_{X, x})$ we have $\Psi^{\prime}_x(g)|(\V_x \otimes_{\OO_{X, x}} R_x) = \Psi_x(\bar{g})$ where $\bar{g}$ is the image of $g$ in $\Gal(\PP(x)/\K_{X, x})$ under the natural map and $\V_x$ is thought of as $\V^{\prime}_x$ via the given isomorphism, and (iii) $\forall x \in \Supp(\PP^{\prime})$ we have $\mu^{\prime}_x|(\V_x \otimes_{\OO_{X, x}} \PP(x)) = \mu_x$. It is denoted by the notation $(\V, \{\Psi_x\}_{x \in \Supp(\PP)}, \{\mu_x\}_{x \in \Supp(\PP)}) \sim (\V^{\prime}, \{\Psi^{\prime}_x\}_{x \in \Supp(\PP^{\prime})}, \{\mu^{\prime}_x\}_{x \in \Supp(\PP^{\prime})})$.    
\end{defn}

\begin{rmk}
For $x \in \Supp(\PP^{\prime})$ but $x \notin \Supp(\PP)$ we take $\Psi_x$ as the trivial map. Note that $\forall x \in X$ we have $R_x \subseteq R^{\prime}_x, \PP(x) \subseteq \PP^{\prime}_x$, hence $\V_x \otimes_{\OO_{X, x}} R_x \subseteq \V^{\prime}_x \otimes_{\OO_{X, x}} R^{\prime}_x, \V_x \otimes_{\OO_{X, x}} \PP(x) \subseteq \V^{\prime}_x \otimes_{\OO_{X, x}} \PP^{\prime}(x)$ (via the isomorphism $\V \cong \V^{\prime}$).
 
\end{rmk}

Now given branch data $\PP \leq \PP^{\prime}$ on $X$ and $(\V, \Psi, \mu) \in \PV(X, \PP)$ we would like to construct a parabolic bundle with branch data $\PP^{\prime}$ which is equivalent to the given one. We take the same underlying vector bundle for the new parabolic bundle i.e $\V$. For any $x \in X$, let $R_x, R^{\prime}_x$ denote the integral closure of $\hat{\OO}_{X, x}$ in $\PP(x), \PP^{\prime}_x$ respectively. Denote the natural action of $\Gal(\PP^{\prime}(x)/\K_{X, x})$ on $R^{\prime}_x$ by $\psi^{\prime}_x$ i.e. $\psi^{\prime}_x : \Gal(\PP^{\prime}(x)/\K_{X, x}) \to \Aut_{Ring}(R^{\prime}_x)$. Note that
\begin{align*}
\V_x \otimes_{\OO_{X, x}} R^{\prime}_x \cong (\V_x \otimes_{\OO_{X, x}} R_x) \otimes_{R_x} R^{\prime}_x \\
\V_x \otimes_{\OO_{X, x}} \PP^{\prime}(x) \cong (\V_x \otimes_{\OO_{X, x}} \PP(x)) \otimes_{\PP(x)} \PP^{\prime}(x) 
\end{align*} 
We have the following natural group homomorphism $\Gal(\PP^{\prime}(x)/\K_{X, x}) \twoheadrightarrow \Gal(\PP(x)/\K_{X, x})$. Define $\Psi^{\prime}_x : \Gal(\PP^{\prime}(x)/\K_{X, x}) \to \Aut_{Ab}(\V_x \otimes_{\OO_{X, x}} R^{\prime}_x)$ by 
\begin{align*}
\Psi^{\prime}_x(g) (v \otimes s) = \Psi_x(\bar{g}) (v) \otimes_{R_x} \psi^{\prime}_x(g)(s)
\end{align*}
for every $g \in \Gal(\PP^{\prime}(x)/\K_{X, x}), v \in \V_x \otimes_{\OO_{X, x}} R_x, s \in R^{\prime}_x$ where $\bar{g}$ is the image of $g$ in $\Gal(\PP(x)/\K_{X, x})$. Similarly we define $\mu^{\prime}_x = \mu_x \otimes_{\PP(x)} Id_{\PP^{\prime}(x)}$ (via the isomorphism stated above).
\begin{lem}
$(\V, \{\Psi^{\prime}_{x}\}_{x \in \Supp(\PP^{\prime})}, \{\mu^{\prime}_{x}\}_{x \in \Supp(\PP^{\prime})})$ as defined above is an algebraic parabolic bundle on $X$ with branch data $\PP^{\prime}$. 
\end{lem}

\begin{proof}
The linearity condition for $\Psi^{\prime}_x$ follows immediately from the definition. For the patching condition we need to prove that for any $g \in \Gal(\PP^{\prime}(x)/\K_{X, x})$
\begin{align*}
\mu^{\prime}_x \circ (Id_{V_x} \otimes \psi^{\prime 0}_x(g)) = \Psi^{\prime 0}_x(g) \circ \mu^{\prime}_x.
\end{align*}
An arbitrary element of $\V_x \otimes_{\OO_{X, x}} \PP^{\prime}(x)$ may be written as finite sum of elements of the form $(v \otimes 1) \otimes s$ for some $v \in \V_x, s \in R^{\prime}_x, 1 \in \PP(x)$. Then we have
\begin{align*}
\mu^{\prime}_x \circ (Id_{V_x} \otimes \psi^{\prime 0}_x(g))((v \otimes 1) \otimes s) = \mu_x(v \otimes 1) \otimes \psi^{\prime 0}_x(s) = \mu_x \circ (Id_{V_x} \otimes \psi^{0}_x(\bar{g}))(v \otimes 1) \otimes \psi^{\prime 0}_x(s) \\
= \Psi^{0}_x(\bar{g})(\mu_x(v \otimes 1)) \otimes \psi^{\prime 0}_x(s) = \Psi^{\prime 0}_x(g)(\mu_x(v \otimes 1) \otimes s) = \Psi^{\prime 0}_x(g) \circ \mu^{\prime}_x ((v \otimes 1) \otimes s).
\end{align*}
Hence we have the required patching condition and the statement is true.
\end{proof}
We denote this bundle by $\imath^*(\V, \{\Psi_x\}_{x \in \Supp(\PP)}, \allowbreak \{\mu_x\}_{x \in \Supp(\PP)})$ where $\imath : (X, \PP^{\prime}) \to (X, \PP)$ is the map induced by $Id_X$.

\begin{prop} \label{pullbackequiv}
Let $\PP, \PP^{\prime}$ be two branch data on $X$ such that $\PP \leq \PP^{\prime}$. Let $(\V, \{\Psi_x\}_{x \in \Supp(\PP)}, \allowbreak \{\mu_x\}_{x \in \Supp(\PP)})$ and $(\V^{\prime}, \{\Psi^{\prime}_x\}_{x \in \Supp(\PP^{\prime})}, \{\mu^{\prime}_x\}_{x \in \Supp(\PP^{\prime})})$ be two algebraic parabolic bundles on $X$ with branch data $\PP$ and $\PP^{\prime}$ respectively. Let $\imath : (X, \PP^{\prime}) \to (X, \PP)$ be the morphism of formal orbifolds induced by $Id_X$. Then $(\V, \{\Psi_x\}_{x \in \Supp(\PP)}, \allowbreak \{\mu_x\}_{x \in \Supp(\PP)}) \sim (\V^{\prime}, \{\Psi^{\prime}_x\}_{x \in \Supp(\PP^{\prime})}, \allowbreak \{\mu^{\prime}_x\}_{x \in \Supp(\PP^{\prime})})$ iff $\imath^*(\V, \{\Psi_x\}_{x \in \Supp(\PP)}, \allowbreak \{\mu_x\}_{x \in \Supp(\PP)}) \cong (\V^{\prime}, \{\Psi^{\prime}_x\}_{x \in \Supp(\PP^{\prime})}, \{\mu^{\prime}_x\}_{x \in \Supp(\PP^{\prime})})$.
\end{prop}

\begin{proof}
First we note that in both of the above situations we have $\V \cong \V^{\prime}$. Hence to simplify notation throughout this proof we would treat these two isomorphic bundles as the same bundle. 

For the 'if' part, by our hypothesis,  for any $x \in \Supp(\PP^{\prime}), g \in \Gal(\PP^{\prime}(x)/\K_{X, x})$ we have $\Psi^{\prime}_x(g)(v \otimes r^{\prime}) = \Psi_x(\bar{g}) \otimes_{R_x} \psi^{\prime}_x(g)(r^{\prime})$ where $v \in \V_x \otimes_{\OO_{X, x}} R_x, r^{\prime} \in R^{\prime}_x$ and $\bar{g}$ is the image of $g$ in $\Gal(\PP(x)/\K_{X, x})$. Then it immediately follows that $\Psi^{\prime}_x(g)|(\V_x \otimes_{\OO_{X, x}} R_x) = \Psi_x(\bar{g})$. We also have $\mu'_x = \mu_x \otimes_{\PP(x)} Id_{\PP'(x)} \Rightarrow \mu'_x|(\V_x \otimes_{\OO_{X, x}} \PP(x)) = \mu_x$. Hence the required equivalence holds.

For the 'only if' part, we have
\begin{align*}
\Psi^{\prime}_x(g)(v \otimes r^{\prime}) = \Psi^{\prime}_x(g)(r^{\prime} \cdot v \otimes 1) = \psi^{\prime}_x(g)(r^{\prime})\Psi^{\prime}_x(g)(v \otimes 1) \\ 
= \psi^{\prime}_x(g)(r^{\prime})(\Psi_x(\bar{g})(v) \otimes 1) = \Psi_x(\bar{g}) \otimes \psi^{\prime}_x(g)(r^{\prime})
\end{align*}
(note we are using that $\V_x \otimes_{\OO_{X, x}} R^{\prime}_x \cong (\V_x \otimes_{\OO_{X, x}} R_x) \otimes_{R_x} R^{\prime}_x$). Similarly 
\begin{align*}
\mu'_x(u \otimes \alpha') = \alpha'\mu'_x(u \otimes 1) = \alpha'(\mu_x(u) \otimes 1) = \mu_x(u) \otimes \alpha'
\end{align*}
where $u \in \V_x \otimes_{\OO_{X, x}} \PP(x), \alpha' \in \PP'(x)$. Clearly this proves that $\imath^*(\V, \{\Psi_x\}_{x \in \Supp(\PP)}, \allowbreak \{\mu_x\}_{x \in \Supp(\PP)}) \allowbreak \cong (\V, \{\Psi^{\prime}_x\}_{x \in \Supp(\PP^{\prime})}, \{\mu^{\prime}_x\}_{x \in \Supp(\PP^{\prime})})$.
\end{proof}

\begin{prop} \label{geomequiv}
Let $(X, \PP)$ and $(X, \PP^{\prime})$ be geometric formal orbifolds such that $\PP \leq \PP^{\prime}$. Let $(Y, 0)$ and $(Y^{\prime}, 0)$ be their respective connected Galois \'etale covers with Galois groups $G, G^{\prime}$ which fit into the following commutative diagram of fromal orbifolds
\[
\xymatrix{
(Y, 0) \ar[d]^{f} & (Y^{\prime}, 0) \ar[d]^{f^{\prime}} \ar[l]^{\tilde{\imath}} \\
(X, \PP) & (X, \PP^{\prime}) \ar[l]^{\imath}  
}
\]
where the map $\imath$ is induced by $Id_X$. Let $(\V, \{\Psi_x\}_{x \in \Supp(\PP)}, \{\mu_x\}_{x \in \Supp(\PP)})$ and $(\V^{\prime}, \{\Psi^{\prime}_x\}_{x \in \Supp(\PP^{\prime})}, \allowbreak \{\mu^{\prime}_x\}_{x \in \Supp(\PP^{\prime})})$ be two algebraic parabolic bundles on $X$ with branch data $\PP$ and $\PP^{\prime}$ respectively. Let $\E = \TT_{(X, \PP)}^Y(\V)$ and $\E^{\prime} = \TT_{(X, \PP^{\prime})}^{Y^{\prime}}(\V^{\prime})$. Then $(\V, \{\Psi_x\}_{x \in \Supp(\PP)}, \{\mu_x\}_{x \in \Supp(\PP)}) \sim (\V^{\prime}, \{\Psi^{\prime}_x\}_{x \in \Supp(\PP^{\prime})}, \allowbreak \{\mu^{\prime}_x\}_{x \in \Supp(\PP^{\prime})})$ iff $\E^{\prime} \cong \tilde{\imath}^*\E$ as $G^{\prime}$-bundles.
\end{prop}

\begin{proof}
First we note that in the given situation $\tilde{\imath}$ is a $H$-Galois cover of curves where $H$ is a normal subgroup of $G^{\prime}$ such that $G^{\prime}/H \cong G$. Assume $\E^{\prime} \cong \tilde{\imath}^*\E$ as $G^{\prime}$-bundles. Then as vector bundles $\V^{\prime} \cong (f^{\prime}_*\E^{\prime})^{G^{\prime}} \cong (f_*\tilde{\imath}_*\tilde{\imath}^*\E)^{G^{\prime}} \cong (f_*\E)^G \cong \V$. Moreover $(\tilde{\imath}_*\E^{\prime})^H \cong \E$ as $G$-bundles. Hence for any $x \in \Supp(\PP^{\prime})$ we must have $(\V^{\prime}_x \otimes_{\OO_{X, x}} R^{\prime}_x)^{H_x} \cong (\V_x \otimes_{\OO_{X, x}} R_x)$ where $H_x = \Gal(\PP^{\prime}(x)/\PP(x))$. It follows that $\Psi^{\prime}_x(g)|(\V_x \otimes_{\OO_{X, x}} R_x) = \Psi_x(\bar{g})$ for any $g \in G^{\prime}_x$ and $\bar{g}$ is the image of $g$ in $G_x$ (note that $\V_x$ is thought of as $\V^{\prime}_x$ via the isomorphism above). Keeping in mind Theorem \ref{formalgluing}, $\E^{\prime} \cong \tilde{\imath}^*\E$ also gives us the condition on $\mu, \mu^{\prime}$. In other words $\V$ is equivalent to $\V^{\prime}$.

Conversely assume that $\V$ and $\V^{\prime}$ are equivalent. We need to show that $\E^{\prime} \cong \tilde{\imath}^*\E$ as $G^{\prime}$-bundles. Let $\tilde{\V}$ be the parabolic bundle on $X$ with branch data $\PP^{\prime}$ corresponding to the $G^{\prime}$-bundle $\tilde{\E} := \tilde{\imath}^*\E$ on $Y^{\prime}$ i.e. $\tilde{\V} = \SSS_{(X, \PP^{\prime})}^{Y^{\prime}}(\tilde\E)$. Then it suffices to prove that $\tilde{\V} \cong \V^{\prime}$ as parabolic bundles. 
 
Now as vector bundles $\tilde{\V} \cong (f^{\prime}_*\tilde{\imath}^*\E)^{G^{\prime}} \cong (f_*\E)^G \cong \V \cong \V^{\prime}$. For $x\in \Supp(\PP')$, let $y'\in Y'$ be a point lying above $x$ and $y\in Y$ be the image of $y'$ in $Y$. Note that $\PP'(x)=\K_{Y',y'}$ and $\PP(x)=\K_{Y,y}$. By the construction of parabolic bundles from $G$-bundles we have $\hat \E_y \cong \V_x\otimes_{\OO_{X,x}} \hat \OO_{Y,y}$ (see \ref{compidentity}). Also $\widehat{(\tilde{\imath}^*\E)}_{y'}\cong \hat\E_y\otimes_{\hat\OO_{Y,y}}\hat \OO_{Y',y'}\cong\V_x\otimes_{\OO_{X,x}} \hat \OO_{Y,y}\otimes_{\hat\OO_{Y,y}}\hat \OO_{Y',y'}\cong \hat \E'_{y'}$.

Note that $\tilde{\E}$ is a also $G'$-bundle on $Y'$ (\cite[Lemma 3.3]{KP}). Moreover the proof of this lemma shows that if $\Lambda$ defines the $G$-bundle structure of $\E$ on $Y$ then $q^*(\Lambda)$ defines the $G'$-bundle structure of $\tilde\imath^*\E$ on $Y'$ where $q:G'\times Y' \to G\times Y$ is the induced map from $G'\to G$ and $Y'\to Y$. Translating this in terms of rings and modules imply that the action $\tilde{\Psi}_x(g)$ for $g\in G'_x$ on $\hat\E'_{y'}\cong \hat\E_y\otimes_{\hat\OO_{Y,y}}\hat\OO_{Y',y'}$ is given by $\Psi_x(\bar{g}) \otimes_{R_x} \psi^{\prime}_x(g)$ where $\psi^{\prime}_x$ denotes the action of $G'_x$ on $R'_x$. Thus $\tilde{\Psi}_x = \Psi'_x$. Similarly we get $\tilde{\mu}_x = \mu_x \otimes_{\PP(x)} Id_{\PP'(x)}$ and hence $\tilde{\mu}_x = \mu'_x$. This implies that $\tilde \V$ is isomorphic to $\V'$.  
\end{proof}

\begin{rmk}
The diagram in the above Proposition may be obtained as follows: let $(Y, O) \to (X, \PP)$ and $(\tilde{Y}, O) \to (X, \PP^{\prime})$ be connected Galois \'etale covers of the respective formal orbifolds. Let $Y^{\prime}$ be a connected component of the normalization of $Y \times_X \tilde{Y}$. Then one can check that $(Y^{\prime}, O) \to (X, \PP^{\prime})$ is a Galois \'etale cover and we have a commutative diagram of formal orbifolds as in the Proposition. 
\end{rmk}

\begin{cor} 
Let us be in the situation of Proposition \ref{geomequiv}. Then $(\V^{\prime}, \{\Psi^{\prime}_x\}_{x \in \Supp(\PP^{\prime})}, \{\mu^{\prime}_x\}_{x \in \Supp(\PP^{\prime})}) \allowbreak \cong \imath^*(\V, \{\Psi_x\}_{x \in \Supp(\PP)}, \{\mu_x\}_{x \in \Supp(\PP)})$ iff $\E^{\prime} \cong \tilde{\imath}^*\E$ as $G^{\prime}$-bundles. In other words, $(\V^{\prime}, \Psi^{\prime}, \mu^{\prime}) \cong \imath^*(\V, \Psi, \mu)$ iff $\TT_{(X, \PP^{\prime})}^{Y^{\prime}}(\V^{\prime}) \cong \tilde{\imath}^*\TT_{(X, \PP)}^Y(\V)$. 
\end{cor}

\begin{cor} \label{pullbackg}
Let us be in the situation of Proposition \ref{geomequiv}. Then $\TT_{(X, \PP^{\prime})}^{Y^{\prime}} \circ \imath^*(\V) \cong \tilde{\imath}^* \circ \TT_{(X, \PP)}^Y(\V)$ for any $\V \in \PV(X, \PP)$. Equivalently, $\imath^* \circ \SSS_{(X, \PP)}^Y(\E) \cong \SSS_{(X, \PP^{\prime})}^{Y^{\prime}} \circ \tilde{\imath}^*(\E)$ for any $\E \in \Vect_G(Y)$.  
\end{cor}

\begin{prop} \label{functpullback}
Let $\imath : (X, \PP^{\prime}) \to (X, \PP)$ be a morphism of formal orbifolds. Then the operation $\imath^*$ defines a functor $\imath^* : \PV(X, \PP) \to \PV(X, \PP^{\prime})$. Moreover, given another morphism of formal orbifolds $\jmath : (X, \PP^{''}) \to (X, \PP^{\prime})$ we have a natural isomorphism of functors $\jmath^*\imath^* \cong (\imath  \jmath)^* : \PV(X, \PP) \to \PV(X, \PP^{''})$.
\end{prop}

\begin{proof}
We need to define the functor on morphisms. Let $(\theta, \{\sigma_x\}_{x \in \Supp(\PP)})$ be a morphism in $\PV(X, \PP)$. Define $\imath^*(\theta, \{\sigma_x\}_{x \in \Supp(\PP)})$ to be the morphism $(\theta, \{\sigma^{\prime}_x\}_{x \in \Supp(\PP^{\prime})})$ where $\sigma^{\prime}_x = \sigma_x \otimes_{R_x} Id_{R^{\prime}_x}$ for any $x \in \Supp(\PP^{\prime})$. Then it is straightforward to check that $(\theta, \{\sigma^{\prime}_x\}_{x \in \Supp(\PP^{\prime})})$ defines a morphism between two objects of $\PV(X, \PP^{\prime})$ obtained by applying $\imath^*$ on two objects of $\PV(X, \PP)$. The rest of the properties are also routine to check.
\end{proof}

\begin{cor}
Let us be in the situation of Proposition \ref{geomequiv}. Then the functor $\imath^* : \PV(X, \PP) \to \PV(X, \PP^{\prime})$ is an embedding.
\end{cor}

\begin{proof}
From what we have already proved we know that $(\V^{\prime}, \{\Psi^{\prime}_x\}_{x \in \Supp(\PP^{\prime})}, \{\mu^{\prime}_x\}_{x \in \Supp(\PP^{\prime})}) \allowbreak \cong \imath^*(\V, \{\Psi_x\}_{x \in \Supp(\PP)}, \{\mu_x\}_{x \in \Supp(\PP)})$ iff $\E^{\prime} \cong \tilde{\imath}^*\E$ as $G^{\prime}$-bundles. Then we would have $\imath^*(\V_1, \Psi_{1}, \mu_{1}) \cong \imath^*(\V_2, \Psi_{2}, \mu_{2})$ iff $\tilde{\imath}^*\E_1 \cong \tilde{\imath}^*\E_2$ as $G^{\prime}$-bundles. But by our choice $\tilde{\imath}$ is a $H$-Galois cover. Hence $\tilde{\imath}^*\E_1 \cong \tilde{\imath}^*\E_2 \Rightarrow \E_1 \cong \E_2$ and the statement is true.
\end{proof}

Now we generalize the definition of equivalence of two parabolic bundles even if their branch data are not comparable.
\begin{defn}
Let $\PP, \PP^{\prime}$ be two branch data on $X$. Let $(\V, \{\Psi_x\}_{x \in \Supp(\PP)}, \{\mu_x\}_{x \in \Supp(\PP)})$ and $(\V^{\prime}, \{\Psi^{\prime}_x\}_{x \in \Supp(\PP^{\prime})}, \{\mu^{\prime}_x\}_{x \in \Supp(\PP^{\prime})})$ be two algebraic parabolic bundles on $X$ with branch data $\PP$ and $\PP^{\prime}$ respectively. We say that these two algebraic parabolic bundles are equivalent if there exists a branch data $\QQ$ on $X$ such that $\PP \leq \QQ, \PP^{\prime} \leq \QQ$ and under the functors $\imath^* : \PV(X, \PP) \to \PV(X, \QQ)$ and $\imath^{\prime*} : \PV(X, \PP^{\prime}) \to \PV(X, \QQ)$ the respective images are isomorphic. We again use $\sim$ to denote equivalence of two parabolic bundles.  
\end{defn}

\begin{lem}
The relation $\sim$ between parabolic bundles with branch data defined on $X$ is an equivalence relation.
\end{lem}

\begin{proof}
We only need to check transitivity. Suppose we are given branch data $\PP_1, \PP_2, \PP_3$ on $X$ and parabolic bundles $\V_i$ with branch data $\PP_i$ for $i = 1, 2, 3$. Assume that $\V_1 \sim \V_2$ and $\V_2 \sim \V_3$. Then by definition we have branch data $\QQ_i, i = 1, 2$, on X such that $f_1^*\V_1 \cong f_2^*\V_2$ and $g_2^*\V_2 \cong g_3^*\V_3$ where $f_i : (X, \QQ_1) \to (X, \PP_i),i= 1, 2$, and $g_j : (X, \QQ_2) \to (X, \PP_j), j = 2, 3$, are the natural maps between formal orbifolds. Let $\RR = \QQ_1\QQ_2$ be the compositum of these two branch data. We have maps $h_i : (X, \RR) \to (X, \QQ_i), i = 1, 2$, of formal orbifolds such that $f_2 \circ h_1 = g_2 \circ h_2$. Hence $h_1^*f_2^*\V_2 \cong h_2^*g_2^*\V_2 \Rightarrow h_1^*f_1^*\V_1 \cong h_2^*g_3^*\V_3$. Hence by our definition $\V_1 \sim \V_3$ and we are done. 
\end{proof}

\begin{defn}
A parabolic bundle on a smooth projective curve $X$ is an equivalence class under $\sim$ of algebraic parabolic bundles on $X$ with some branch data. We denote this category by $\PV(X)$.
\end{defn}

From our definition of equivalence it is clear that the trivial parabolic bundle on $X$ corresponds to the equivalence class of $\OO_X^{\oplus n}$ with trivial branch data $O$. The trivial parabolic bundle with branch data $\PP$ obviously belongs to this equivalence class.

\begin{prop}
The category $\PV(X)$ is equivalent to the category of "orbifold bundles" on $X$ as defined in \cite{KP}, Definition 3.8. 
\end{prop}

\begin{proof}
As any branch data can be dominated by a geometric branch data, it is clear that an object in $\PV(X)$ can be represented by an algebraic parabolic bundle over $X$ with geometric branch data, say $\PP$. But we have already seen that $\PV(X, \PP)$ is equivalent to $\Vect_G(Y)$ where $(Y, O)$ is a connected Galois \'etale cover of $(X, \PP)$ with Galois group $G$. This implies the desired equivalence of categories. 
\end{proof}

\begin{prop}
Let characteristic of $k$ be zero. Then the category $\PV(X)$ is equivalent to the category of ``parabolic bundle with rational weights'' on $X$ according to the definition in \cite{MS}. 
\end{prop}

\begin{proof}
For the sake of convenience we start with the case when we have a parabolic bundle $\V$ with parabolic structure supported on just a single point $\{p\}$. By our definition we have $\Psi : I \to \Aut_{\Ab}(\V_p \otimes_{\OO_{X, p}} R)$. Note that $I$ is a cyclic group and $R$ is a complete local ring with the residue field being $k$. So we get an induced map $\bar{\Psi} : I \to \Aut_{\Ab}(\V_p \otimes_{\OO_{X, p}} k)$. The linearity condition on $\Psi$ implies that $\bar{\Psi}$ is $k$-linear (by our definition the map induced by $\psi$ is $Id$ on $k$). Note that $\V_p \otimes_{\OO_{X, p}} k$ is nothing but the fibre of the vector bundle $\V$ at $p$ which we denote by $\V(p)$. Let $I = \langle \gamma \rangle$. Then we may choose a basis for $\V(p)$ consisting of eigenvectors of $\bar{\Psi}(\gamma)$. Clearly the eigenvalues of $\bar{\Psi}(\gamma)$ are $N$-th roots of unity where $N = |I|$ i.e. of the form $exp(2\pi\sqrt{-1}a/N)$, $0 \leq a \leq N, a \in \Z$. We can arrange these eigenvalues such that the integers $a$ are in increasing order . From this we can easily construct a decreasing filtration of $\V(p)$ along with weights given by the rational numbers $a/N$. Thus we recover the classical definition of parabolic bundle as given in \cite{MS}. In the general case we just apply the above procedure for each point in the support of the branch data.

Now starting from a "parabolic bundle with rational weights", first we write the weights as elements of $\frac{1}{N}\Z$ for some integer $N$. It is well known that we can construct Galois cover $Y \xrightarrow{f} X$ such that the branch locus $B$ contains the parabolic divisor and ramification index at each ramification point is $N$. Now we can treat the original parabolic bundle as having divisor $B$ by adding some points to the original parabolic divisor with trivial parabolic structure at those points. We know that there is an equivalence between the category of parabolic bundles with weights lying in $\frac{1}{N}\Z$ and parabolic divisor $B$ and the category of $\Gamma$-bundles on $Y$ where $\Gamma = \Gal(Y/X)$. But Theorem \ref{genequival} gives us that $\Vect_{\Gamma}(Y)$ is equivalent to $\PV(X, B_f)$ where $B_f$ is the branch data associated to $f$. Thus we get an element of $\PV(X)$. One easily checks that these two functors are inverse to each other.
\end{proof}

In what follows we will define pullback, tensor product, dual and pushforward operations on $\PV(X)$. We will also show that the functors defining equivalence between $\PV(X)$ and orbifold bundles on $X$ are well behaved with respect to these operations. 

\subsection{Pullback} \label{subsec:Pullback}

We would like to construct a pullback operation for parabolic bundles i.e. given a cover $f : Y \to X$ of curves, we want a functor $f^* : \PV(X) \to \PV(Y)$. Let $\V \in \PV(X)$ and we may assume, without loss of generality, that the associated branch data $\PP$ is $\geq B_f$. Here $B_f$ is the branch data associated to the map $f : Y \to X$. Then clearly the induced morphism $(Y, f^*\PP) \xrightarrow{f} (X, \PP)$ is \'etale. Now we can get a Galois \'etale cover $(Z, O) \to (X, \PP)$ such that it factors through $(Y, f^*\PP) \to (X, \PP)$. In other words we have a commutative diagram of formal orbifolds
\[
\xymatrix{
& (Z, O) \ar[d]^{g} \ar[ld]_{h} \\
(Y, f^*\PP) \ar[r]^{f}  & (X, \PP)
}
\]
where $g$ and $h$ are $G$-Galois \'etale and $H$-Galois \'etale covers respectively. Now we define $f^*\V$ as the equivalence class of the parabolic bundle $\SSS_{(Y, f^*\PP)}^Z \circ \TT_{(X, \PP)}^Z(\V)$. Of course we need to check that this operation is well defined. 

Firstly, if we made another choice of Galois \'etale cover $(Z^{'}, O) \to (X, \PP)$ factoring through $f$, then we need to show that $\SSS_{(Y, f^*\PP)}^Z \circ \TT_{(X, \PP)}^Z(\V) \cong \SSS_{(Y, f^*\PP)}^{Z^{'}} \circ \TT_{(X, \PP)}^{Z^{'}}(\V)$. Now we can construct a Galois \'etale cover $(\tilde{Z}, O) \to (X, \PP)$ which dominates both $(Z, O)$ and $(Z^{'}, O)$. Then we have
\begin{align*}
\TT_{(X, \PP)}^{\tilde{Z}} = \TT_{(Z, O)}^{\tilde{Z}} \circ \TT_{(X, \PP)}^{Z} \Rightarrow \SSS_{(Y, f^*\PP)}^{\tilde{Z}} \circ \TT_{(X, \PP)}^{\tilde{Z}} & = \SSS_{(Y, f^*\PP)}^{\tilde{Z}} \circ \TT_{(Z, O)}^{\tilde{Z}} \circ \TT_{(X, \PP)}^{Z} \\
& = \SSS_{(Y, f^*\PP)}^{Z} \circ \TT_{(X, \PP)}^{Z}
\end{align*}
Similarly we get $\SSS_{(Y, f^*\PP)}^{\tilde{Z}} \circ \TT_{(X, \PP)}^{\tilde{Z}} = \SSS_{(Y, f^*\PP)}^{Z^{'}} \circ \TT_{(X, \PP)}^{Z^{'}}$ and we have the required isomorphism.

Secondly, if we have $(\V, \Psi, \mu) \sim (\V^{\prime}, \Psi^{\prime}, \mu^{\prime})$, where $\V'$ has branch data $\PP' \geq B_f$, then we need to show that $\SSS_{(Y, f^*\PP)}^Z \circ \TT_{(X, \PP)}^Z(\V) \sim \SSS_{(Y, f^*\PP^{\prime})}^{Z^{\prime}} \circ \TT_{(X, \PP^{\prime})}^{Z^{\prime}}(\V^{\prime})$ where $Z^{\prime}$ has been chosen in the same manner as $Z$. Choose branch data $\QQ$ on $X$ such that $\QQ \geq \PP, \PP^{\prime}$ and $\imath^*(\V, \Psi, \mu) \cong \imath^{\prime*}(\V^{\prime}, \Psi^{\prime}, \mu^{\prime})$ where $\imath^*, \imath^{\prime*}$ are the functors described in Proposition \ref{functpullback}. Now we can get a commutative diagram :
\[
\xymatrix{
(Z, O) \ar[d] & (\tilde{Z}, O) \ar[l]_{\tilde{\imath}} \ar[d] \ar[r]^{\tilde{\imath}^{\prime}} & (Z^{\prime}, O) \ar[d] \\
(Y, f^*\PP) \ar[d] & (Y, f^*\QQ) \ar[l]_{\jmath} \ar[d] \ar[r]^{\jmath^{\prime}} & (Y, f^*\PP^{\prime}) \ar[d] \\
(X, \PP) & (X, \QQ) \ar[l]_{\imath} \ar[r]^{\imath^{\prime}} & (X, \PP^{\prime})
}
\]
where $(Z^{\prime}, O) \to (X, \PP^{\prime})$ and $(\tilde{Z}, O) \to (X, \QQ)$ is chosen in the same way as $(Z, O) \to (X, \PP)$. Moreover, like Proposition \ref{geomequiv} we can choose $(\tilde{Z}, O)$ such that $\tilde{\imath}$ and $\tilde{\imath}^{\prime}$ are Galois covers. Clearly $f^*\QQ \geq f^*\PP, f^*\PP^{\prime}$. Hence it suffices to show that
\begin{align*}
& \jmath^* \circ \SSS_{(Y, f^*\PP)}^Z \circ \TT_{(X, \PP)}^Z(\V) \cong \jmath^{\prime*} \circ \SSS_{(Y, f^*\PP^{\prime})}^{Z^{\prime}} \circ \TT_{(X, \PP^{\prime})}^{Z^{\prime}}(\V^{\prime}) \\
& \Leftrightarrow \SSS_{(Y, f^*\QQ)}^{\tilde{Z}} \circ \tilde{\imath}^* \circ \TT_{(X, \PP)}^Z(\V) \cong \SSS_{(Y, f^*\QQ)}^{\tilde{Z}} \circ \tilde{\imath}^{\prime*} \circ \TT_{(X, \PP^{\prime})}^{Z^{\prime}}(\V^{\prime}) & \text{(by Corollary \ref{pullbackg})} \\
& \Leftrightarrow \tilde{\imath}^* \circ \TT_{(X, \PP)}^Z(\V) \cong \tilde{\imath}^{\prime*} \circ \TT_{(X, \PP^{\prime})}^{Z^{\prime}}(\V^{\prime}) & \text{(as $\SSS_{(Y, f^*\QQ)}^{\tilde{Z}}$ is an equivalence)} \\
& \Leftrightarrow \TT_{(X, \QQ)}^{\tilde{Z}} \circ \imath^*(\V) \cong \TT_{(X, \QQ)}^{\tilde{Z}} \circ \imath^{\prime*}(\V^{\prime}) & \text{(by Corollary \ref{pullbackg})} \\
& \Leftrightarrow \imath^*(\V) \cong \imath^{\prime*}(\V^{\prime}) & \text{(as $\TT_{(X, \QQ)}^{\tilde{Z}}$ is an equivalence)}
\end{align*}
Hence we are in good shape and we can make the following definition
\begin{defn}
Let $ : Y \to X$ be a covering map between smooth projective curves. Then for any $\V \in \PV(X)$ the pullback bundle $f^*\V \in \PV(Y)$ is defined as the equivalence class of $\SSS_{(Y, f^*\PP)}^Z \circ \TT_{(X, \PP)}^Z(\V)$ where $\PP, Z$ has been chosen as above.
\end{defn}

\begin{rmk}
Let $\V \in \PV(X, \PP)$ where $\PP$ is a geometric branch data on $X$ and let $f : (Y, O) \to (X, \PP)$ be a $G$-Galois \'etale cover. Let $\E = \TT_{(X, \PP)}^Y(\V)$. Then by our construction, we have a natural isomorphism $f^*\V \cong \E$ as $G$-bundles.
\end{rmk}

\begin{prop}
Let $f : Y \to X$ be a covering map of smooth projective curves. Then we have a functor $f^* : \PV(X) \to \PV(Y)$. We also have $f^*(\OO_X^{\oplus n}) = \OO_Y^{\oplus n}$. Moreover given another cover $g : Z \to Y$ we have a natural isomorphism of functors $(fg)^* \cong g^*f^*$. 
\end{prop}

\begin{proof}
Given two objects of $\PV(X)$, we can always choose representatives for them with the same branch data. Thus a morphism in $\PV(X)$ can be represented by a morphism in $\PV(X, \PP)$ for some branch data $\PP$. So we define the functor $f^*$ for morphisms in the same way as defined for objects i.e. by the equivalence class of $\SSS_{(Y, f^*\PP)}^Z \circ \TT_{(X, \PP)}^Z$ applied to the morphism where $(Z, O)$ is a Galois \'etale cover of $(X, \PP)$ chosen as above. The rest of the details are left to the reader for verification.   
\end{proof}

\subsection{Tensor Operation}

First we give the definition when the parabolic bundles have the same branch data.

Let $(X, \PP)$ be a formal orbifold and $(\V_1, \{\Psi_{1x}\}_{x \in \Supp(\PP)}, \{\mu_{1x}\}_{x \in \Supp(\PP)})$ and $(\V_2, \allowbreak \{\Psi_{2x}\}_{x \in \Supp(\PP)}, \allowbreak \{\mu_{2x}\}_{x \in \Supp(\PP)})$ be two objects of $\PV(X, \PP)$. Define $\V = \V_1 \otimes_{\OO_X} \V_2$. We have the following canonical isomorphisms
\begin{align*}
(\V_{1x} \otimes_{\OO_{X, x}} \V_{2x}) \otimes_{\OO_{X, x}} R_x \cong (\V_{1x} \otimes_{\OO_{X, x}} R_x) \otimes_{R_x} (\V_{2x} \otimes_{\OO_{X, x}} R_x) \\
(\V_{1x} \otimes_{\OO_{X, x}} \V_{2x}) \otimes_{\OO_{X, x}} \PP(x) \cong (\V_{1x} \otimes_{\OO_{X, x}} \PP(x)) \otimes_{\PP(x)} (\V_{2x} \otimes_{\OO_{X, x}} \PP(x))
\end{align*}
for any $x \in \Supp(\PP)$. Using them define $\Psi : \Gal(\PP(x)/\K_{X, x}) \to \Aut_{\Ab}((\V_{1x} \otimes_{\OO_{X, x}} \V_{2x}) \otimes_{\OO_{X, x}} R_x)$ by $\Psi_x = \Psi_{1x} \otimes_{R_x} \Psi_{2x}$ and similarly define $\mu_x = \mu_{1x} \otimes_{\PP(x)} \mu_{2x}$.

\begin{lem}
$(\V, \{\Psi_x\}_{x \in \Supp(\PP)}, \{\mu_x\}_{x \in \Supp(\PP)})$ as defined above is an object of $\PV(X, \PP)$.
\end{lem}

\begin{proof}
The required properties follow from those of the given parabolic bundles keeping in mind the natural isomorphisms stated above.
\end{proof}

\begin{defn}
Let $(X, \PP)$ be a formal orbifold and $(\V_1, \{\Psi_{1x}\}_{x \in \Supp(\PP)}, \{\mu_{1x}\}_{x \in \Supp(\PP)})$ and $(\V_2, \allowbreak \{\Psi_{2x}\}_{x \in \Supp(\PP)}, \allowbreak \{\mu_{2x}\}_{x \in \Supp(\PP)})$ be two objects of $\PV(X, \PP)$. Then their tensor product, denoted by $(\V_1, \{\Psi_{1x}\}_{x \in \Supp(\PP)}, \{\mu_{1x}\}_{x \in \Supp(\PP)}) \otimes (\V_2, \{\Psi_{2x}\}_{x \in \Supp(\PP)} \{\mu_{2x}\}_{x \in \Supp(\PP)})$, is defined as the parabolic bundle $(\V, \{\Psi_x\}_{x \in \Supp(\PP)}, \{\mu_x\}_{x \in \Supp(\PP)})$ constructed in the above Lemma.
\end{defn}

\begin{prop} \label{tensorfunct}
Let $(X, \PP)$ be a geometric formal orbifold and let $\V_1, \V_2 \in \PV(X, \PP)$. Let $(Y, O)$ be a connected Galois \'etale cover of $(X, \PP)$ with Galois group G and let $\E_i = \TT_{(X, \PP)}^Y(\V_i)$ for $i = 1, 2$. Then $\TT_{(X, \PP)}^Y(\V_1 \otimes \V_2) \cong \E_1 \otimes \E_2$ as $G$-bundles. Equivalently $\SSS_{(X, \PP)}^Y(\E_1 \otimes \E_2) \cong \V_1 \otimes \V_2$ as parabolic bundles.
\end{prop}

\begin{proof}
The first isomorphism follows from the fact that both the $G$-bundles have the same local data after we take into account the canonical isomorphism used in the definition above. The second isomorphism is obviously equivalent to the first. 
\end{proof}

\begin{prop} \label{tensorequi}
Let $\imath : (X, \QQ) \to (X, \PP)$ be a morphism of formal orbifolds. Then the functor $\imath^* : \PV(X, \PP) \to \PV(X, \QQ)$ commutes with the tensor operation i.e. if $\V_1, \V_2 \in \PV(X, \PP)$ then $\imath^*(\V_1 \otimes \V_2) \cong \imath^*\V_1 \otimes \imath^*\V_2$. 
\end{prop}

\begin{proof}
It is a routine check.
\end{proof}

Now we would like to extend the tensor product operation to the category $\PV(X)$. Given two objects in $\PV(X)$ we can choose representatives $\V_1 \in \PV(X, \PP)$ and $\V_2 \in \PV(X, \PP)$ with the same branch data $\PP$. Then we can consider the equivalence class of $\V_1 \otimes \V_2$ to be the tensor product of the given parabolic bundles. Let $\V_1^{\prime} \in \PV(X, \PP^{\prime})$ and $\V_2^{\prime} \in \PV(X, \PP^{\prime})$ be different representatives for the given parabolic bundles. Then we also have the tensor product $\V_1^{\prime} \otimes \V_2^{\prime}$.
\begin{lem}
$\V_1 \otimes \V_2 \sim \V_1^{\prime} \otimes \V_2^{\prime}$. 
\end{lem}

\begin{proof}
As $\V_1 \sim \V_1^{\prime}$, by definition there is a branch data $\QQ \geq \PP, \PP^{\prime}$ such that $\imath^*\V_1 \cong \jmath^*\V_1^{\prime}$ and $\imath^*\V_2 \cong \jmath^*\V_2^{\prime}$. Here $\imath, \jmath$ are the morphisms inuced by $Id_X$ from $(X, \QQ)$ to $(X, \PP), (X, \PP^{\prime})$. Then the statement follows from Proposition \ref{tensorequi}.
\end{proof}
Using this Lemma we can make the following definition.
\begin{defn}
Let $X$ be a smooth projective curve. Then we define a tensor product of two objects in $\PV(X)$ as follows : let $\V_1 \in \PV(X, \PP)$ and $\V_2 \in \PV(X, \PP)$ be representatives of these two objects. Then the equivalence class of $\V_1 \otimes \V_2$ is defined to be the tensor product of the given objects.
\end{defn}

\begin{prop}
\begin{itemize}
\item[(i)] The tensor product operation defined on $\PV(X)$ is compatible with the tensor product operation on $\PV(X, \PP)$ for any branch data $\PP$.
\item[(ii)] For any covering morphism $f : Y \to X$ of smooth projective curves, the pullback functor $f^* : \PV(X) \to \PV(Y)$ commutes with the tensor product operation.
\end{itemize}
\end{prop}

\begin{proof}
The statement in $(i)$ holds by our construction. The statement in $(ii)$ holds because of Proposition \ref{tensorequi}.
\end{proof}

\subsection{Pushforward}

Given a covering morphism of smooth projective curves $f : Y \to X$ we would like to define a direct image functor $f_* : \PV(Y) \to \PV(X)$ which extends the usual functor $f_* : \Vect(Y) \to \Vect(X)$.

Given a branch data $\QQ$ on $Y$ we can find a branch data $\PP$ on $X$ such that $\PP \geq B_f$ and $f^*\PP \geq \QQ$. In view of this fact, given a object of $\PV(Y)$ we may choose a representative with branch data of the form $f^*\PP$ such that $\PP \geq B_f$. Let $\W \in \PV(Y, f^*\PP)$ be the chosen representative. Choose $(Z, O) \to (X, \PP)$, a Galois \'etale cover with Galois group $G$. Consider the normalization of the fibre product $\tilde{Z} :=
\widetilde{Z \times_X Y}$. Then $\tilde{Z} \to Z$ is \'etale, though not necessarily connected. Moreover $(\tilde{Z}, O) \to (Y, f^*\PP)$ is a Galois \'etale cover of formal orbifolds with Galois group $G$. So we have a commutative diagram of orbifolds
\[
\xymatrix{
(\tilde{Z}, O) \ar[d] \ar[r]^{\tilde{f}} & (Z, O) \ar[d] \\
(Y, f^*\PP) \ar[r]^{f}  & (X, \PP)
}
\]
Note that the usual direct image functor $\tilde{f}_*$ takes $\Vect_G(\tilde{Z})$ to $\Vect_G(Z)$. We define $f_*\W$ as the equivalence class of $\SSS_{(X, \PP)}^Z \circ \tilde{f}_* \circ \TT_{(Y, f^*\PP)}^{\tilde{Z}}(\W)$. Suppose we choose another representative $\W_1$ with branch data $f^*\PP_1$ where $\PP_1 \geq B_f$. As above we choose a Galois \'etale cover $(Z_1, O) \to (X, \PP_1)$ and construct $\tilde{Z_1}$. We have to show that $\SSS_{(X, \PP)}^Z \circ \tilde{f}_* \circ \TT_{(Y, f^*\PP)}^{\tilde{Z}}(\W) \sim \SSS_{(X, \PP)}^{Z_1} \circ \tilde{f_1}_* \circ \TT_{(Y, f^*\PP_1)}^{\tilde{Z_1}}(\W_1)$.

Now we choose a branch data $\QQ \geq \PP, \PP_1$ such that $\jmath^*\W \cong \jmath_1^*\W_1$ where $(Y, f^*\QQ) \xrightarrow{\jmath} (Y, f^*\PP)$ and $(Y, f^*\QQ) \xrightarrow{\jmath_1} (Y, f^*\PP_1)$ are the natural maps. Now we have the following commutative diagram
\[
\xymatrix{
& (\tilde{Z}^{\prime}, O) \ar[rr]^{\tilde{f}^{\prime}} \ar@{-->}'[d][dd] \ar[ld]^{\tilde{\jmath}} & & (Z^{\prime}, O) \ar[ld]^{\tilde{\imath}} \ar[dd] \\
(\tilde{Z}, O) \ar[rr]^>>>>{\tilde{f}} \ar[dd] & & (Z, O) \ar[dd] \\
& (Y, f^*\QQ) \ar@{-->}'[r][rr] \ar@{-->}[ld]^{\jmath} & & (X, \QQ) \ar[ld]^{\imath} \\
(Y, f^*\PP) \ar[rr]_{f} & & (X, \PP)
} 
\]
where $Z^{\prime}$ and $\tilde{Z}^{\prime}$ are chosen in the same way as above with natural maps $\tilde{f}^{\prime}, \imath, \tilde{\imath}, \tilde{\jmath}$. Then
\begin{align*}
\imath^* \circ \SSS_{(X, \PP)}^Z \circ \tilde{f}_* \circ \TT_{(Y, f^*\PP)}^{\tilde{Z}}(\W) \cong \SSS_{(X, \QQ)}^{Z^{\prime}} \circ \tilde{\imath}^* \circ \tilde{f}_* \circ \TT_{(Y, f^*\PP)}^{\tilde{Z}}(\W) \\
\cong \SSS_{(X, \QQ)}^{Z^{\prime}} \circ \tilde{f}^{\prime}_* \circ \tilde{\jmath}^* \circ \TT_{(Y, f^*\PP)}^{\tilde{Z}}(\W) \cong \SSS_{(X, \QQ)}^{Z^{\prime}} \circ \tilde{f}^{\prime}_* \circ \TT_{(Y, f^*\QQ)}^{\tilde{Z}^{\prime}} \circ \jmath^*(\W)
\end{align*}
Now we can get a commutative cube as the one above for $(X, \PP_1)$ with the same choice of $Z^{\prime}, \tilde{Z}^{\prime}$ and in the same way deduce that
\begin{align*}
\imath_1^* \circ \SSS_{(X, \PP)}^{Z_1} \circ \tilde{f_1}_* \circ \TT_{(Y, f^*\PP_1)}^{\tilde{Z_1}}(\W_1) \cong \SSS_{(X, \QQ)}^{Z^{\prime}} \circ \tilde{f}^{\prime}_* \circ \TT_{(Y, f^*\QQ)}^{\tilde{Z}^{\prime}} \circ \jmath_1^*(\W_1)
\end{align*}
As $\jmath^*\W \cong \jmath_1^*\W_1$, clearly we have the required equivalence. So our definition makes sense.

\begin{defn}
Let $ : Y \to X$ be a covering map between smooth projective curves. Then for any $\W \in \PV(Y)$ the direct image bundle $f_*(\W)$ is defined as the equivalence class of $\SSS_{(X, \PP)}^Z \circ \tilde{f}_* \circ \TT_{(Y, f^*\PP)}^{\tilde{Z}}(\W)$ where $\PP$ and $\tilde{f} : \tilde{Z} \to Z$ has been chosen as above.
\end{defn}

\begin{prop}
Let $ : Y \to X$ be a covering map between smooth projective curves. We have a functor $f_* : \PV(Y) \to \PV(X)$. Moreover given another cover $g : Z \to Y$ of curves, we have a natural isomorphism of functors $f_* \circ g_* \cong (f \circ g)_*$.
\end{prop}

\begin{proof}
Functoriality of $f_*$ is clear from the definition. For showing the isomorphism of functors we just need to observe that given any object in $\PV(X, Z)$ we may choose representative with branch data of the form $g^*f^*\PP$ such that $\PP \geq B_f$ and $f^*\PP \geq B_g$. Further we can get a commutative diagram of formal orbifolds
\[
\xymatrix{
(\tilde{\tilde{U}}, O) \ar[r]^{\tilde{g}} \ar[d] & (\tilde{U}, O) \ar[r]^{\tilde{f}} \ar[d] & (U, O) \ar[d] \\
(Z, g^*f^*\PP) \ar[r]^{g} & (Y, f^*\PP) \ar[r]^{f} & (X, \PP)
}
\]
where the vertical maps are all Galois \'etale. Then the statement follows from the definition.
\end{proof}

\begin{prop}
Let $ : Y \to X$ be a covering map between smooth projective curves. Let $\V \in \PV(X)$ and $\W \in \PV(Y)$.
\begin{itemize}
\item[(i)] We have a natural isomorphism 
\begin{align*}
\Hom_{\PV(Y)}(f^*\V, \W) \cong \Hom_{\PV(X)}(\V, f_*\W)
\end{align*}
\item[(ii)] We have a natural isomorphism
\begin{align*}
f_*(f^*\V \otimes \W) \cong \V \otimes f_*\W
\end{align*}
\end{itemize}
\end{prop}

\begin{proof}
We choose a representative for $\V$ with branch data $\PP \geq B_f$ and a representative for $\W$ with representative $f^*\PP$. Then we fix a commutative diagram involving $(Y, f^*\PP) \to (X, \PP)$ as in the definition of the direct image functor. Note that $f^*\V \cong \SSS_{(Y, f^*\PP)}^{\tilde{Z}} \circ \tilde{f}^* \circ \TT_{(X, \PP)}^Z(\V)$ as parabolic bundles (by functoriality of pullback and the fact that for Galois \'etale covers the pullback functor is same as the functor $\TT$). Now 
\begin{align*}
\Hom_{\PV(Y)}(f^*\V, \W) & \cong \Hom_{\PV(Y)}(\SSS_{(Y, f^*\PP)}^{\tilde{Z}} \circ \tilde{f}^* \circ \TT_{(X, \PP)}^Z\V, \W) \\
& \cong \Hom_{\Vect_G(\tilde{Z})}(\tilde{f}^* \circ \TT_{(X, \PP)}^Z\V, \TT_{(Y, f^*\PP)}^{\tilde{Z}}\W) \\
& \cong \Hom_{\Vect_G(Z)}(\TT_{(X, \PP)}^Z\V, \tilde{f}_* \circ \TT_{(Y, f^*\PP)}^{\tilde{Z}}\W) \\
& \cong \Hom_{\PV(X)}(\V, \SSS_{(X, \PP)}^Z \circ \tilde{f}_* \circ \TT_{(Y, f^*\PP)}^{\tilde{Z}}\W) \\
& \cong \Hom_{\PV(X)}(\V, f_*\W)
\end{align*}
Here we have used that the functors $\TT_{(X, \PP)}^Z$ and $\SSS_{(Y, f^*\PP)}^{\tilde{Z}}$ are equivalences. We also used that the adjointness property holds for $G$-equivariant morphisms i.e. given a $G$-equivariant map $\tilde{f} : \tilde{Z} \to Z$, we have a natural isomorphism
\begin{align*}
\Hom_{\Vect_G(\tilde{Z})}(\tilde{f}^*\E, \F) \cong \Hom_{\Vect_G(Z)}(\E, \tilde{f}_*\F)
\end{align*}
Clearly the natural maps $\E \to \tilde{f}_*\tilde{f}^*\E$ and $\tilde{f}^*\tilde{f}_*\F \to \F$ are $G$-equivariant. Hence the same way as in the classical case we have the required isomorphism.

By definition we have
\begin{align*}
f_*(f^*\V \otimes \W) = \SSS_{(X, \PP)}^Z \circ \tilde{f}_* \circ \TT_{(Y, f^*\PP)}^{\tilde{Z}}(f^*\V \otimes \W) \\ \cong \SSS_{(X, \PP)}^Z \circ \tilde{f}_*(\tilde{f}^* \circ \TT_{(X, \PP)}^Z\V \otimes \TT_{(Y, f^*\PP)}^{\tilde{Z}}\W) \\
\cong \SSS_{(X, \PP)}^Z(\TT_{(X, \PP)}^Z\V \otimes \tilde{f}_* \circ \TT_{(Y, f^*\PP)}^{\tilde{Z}}\W) \\
\cong \V \otimes \SSS_{(X, \PP)}^Z \circ \tilde{f}_* \circ \TT_{(Y, f^*\PP)}^{\tilde{Z}}\W \cong \V \otimes f_*\W
\end{align*}
Here we have used that the functors $\TT_{(Y, f^*\PP)}^{\tilde{Z}}$ and $\SSS_{(X, \PP)}^Z$ commute with tensor product (see Proposition \ref{tensorfunct}). Also see Theorem 3.14, \cite{KP}.
\end{proof}

\subsection{Dual}

We construct the dual of a parabolic bundle using the equivalence obtained in Theorem \ref{genequival}. Given a parabolic bundle on $X$, we choose a representative $\V \in \PV(X, \PP)$. Fix a Galois \'etale cover $(Y, O) \to (X, \PP)$. Then we have the corresponding orbifold bundle $\TT_{(X, \PP)}^Y(\V)$ on $Y$. Consider the dual orbifold bundle $\TT_{(X, \PP)}^Y(\V)^*$. We define the equivalence class of $\SSS_{(X, \PP)}^Y(\TT_{(X, \PP)}^Y(\V)^*)$ to be the dual of the given parabolic bundle. As before one easily checks (using Proposition \ref{geomequiv}) that this definition does not depend on our choice of the representative or the Galois \'etale cover.

\begin{defn}
Let $\V$ be a parabolic bundle on $X$. The dual parabolic bundle of $\V$, denoted by $\V^*$, is defined to be the equivalence class of $\SSS_{(X, \PP)}^Y(\TT_{(X, \PP)}^Y(\V)^*)$ where $\PP, Y$ has been chosen as above.  
\end{defn}

\begin{rmk}
Note that the underlying vector bundle of $\V^*$ is the dual of the underlying vector bundle of $\V$ and if $\V$ has a representative in $\PV(X,\PP)$ then there is a natural representative of $\V^*$  in  $\PV(X,\PP)$ as well.
\end{rmk}

Now we would like to figure out what the local data for the dual parabolic bundle looks like in terms of the original one. Let us denote the dual parabolic bundle by $(\V^*, \Psi^*, \mu^*)$. Choose $y \in Y, f(y) = x$ such that $x \in \Supp(\PP)$. We have canonical isomorphisms 
\begin{align*}
(\widehat{\E^*})_y \cong (\hat{\E}_y)^* \cong \Hom_{R_y}(\V_x \otimes_{\OO_{X, x}} R_y, R_y) \cong \V_x^* \otimes_{\OO_{X, x}} R_y
\end{align*}
where $\V_x^*$ is the dual of $\V_x$ as $\OO_{X, x}$-module. Fix a basis $\{v_i\}_{i=1}^n$ of $\V_x$ as an $\OO_{X, x}$-module. We also fix the corresponding dual basis $\{v_i^*\}_{i=1}^n$ of $\V_x^*$. Note that $\{v_i \otimes 1\}_{i=1}^n$ and $\{v_i^* \otimes 1\}_{i=1}^n$ gives $R_y$ bases for $\hat{\E}_y = \V_x \otimes_{\OO_{X, x}} R_y$ and $(\widehat{\E^*})_y = \V_x^* \otimes_{\OO_{X, x}} R_y$ respectively. Let $v = \sum_{i=1}^n a_i(v_i \otimes 1)$ be an arbitrary element of $\V_x \otimes_{\OO_{X, x}} R_y$, where $a_i \in R_y \forall i$, which can be thought of as the column vector $(a_1, \ldots, a_n)^{\tr}$. For any $g \in \Gal(\K_{Y, y}/\K_{X, x})$, let $((\alpha_{ij}^g))$ and $((\beta_{ij}^g))$ be the matrix representatives of semilinear maps $\Psi_x(g)$ and $\Psi_x^*(g)$ respectively with respect to these bases. More precisely $\Psi_x(g)(v)=((\alpha_{ij}^g)) (\psi_x(g)(a_1), \ldots, \psi_x(g)(a_n))^{tr}$ and $\Psi^*_x(g)(f)=((\beta_{ij}^g)) (\psi_x(g)(b_1), \ldots, \psi_x(g)(b_n))^{tr}$ where $f=\sum b_i(v_i^*\otimes 1)$. 

We know that action of the inertia groups on the stalks $\hat{\E}_y$ and $(\widehat{\E^*})_y$ of the $G$-bundles $\E$ and its dual $\E^*$ on $Y$ satisfy the relation $\Psi_x^*(g)(f)(v) = f(\Psi_x(g^{-1})(v))$ where $f \in (\widehat{\E^*})_y, v \in \hat{\E}_y$. Even for the semilinear maps, like the linear case, it follows that $((\beta_{ij}^g)) = \left((\alpha_{ij}^{g^{-1}})\right)^{\tr}$. Denoting $\left((\alpha_{ij}^{g^{-1}})\right)^{\tr}$ by $\Psi_x(g)^*$, the above relation can be written as $\Psi_x^*(g) = \Psi_x(g)^*$. It is easy to check that this relation does not depend on the choice of basis. In a similar fashion we deduce that if $((\mu_{ij}))$ is the matrix corresponding to the $G_x$-equivariant isomorphism $\mu_x$ then the matrix corresponding to $\mu_x^*$ is given by $((\mu_{ij}))^{\tr}$. All these can be summarized as
\begin{prop}
Let $\V$ be a parabolic bundle on $X$. Then the dual parabolic bundle $\V^*$ is same as the equivalence class of the parabolic bundle $(\V^{\vee}, \Psi^*, \mu^*)$ where $\V^{\vee}$ is the dual vector bundle of $\V$ and $\Psi^*$ and $\mu^*$ are as defined above.
\end{prop}

The usual results for dual vector bundles are also true for dual parabolic bundles.
\begin{prop}
Let $\V$ be a parabolic bundle on $X$.
\begin{itemize}
\item[(i)] $\V \otimes \V^* \cong \OO_X^{\oplus n^2}$ where $n = \rk(\V)$.
\item[(ii)] $(\V^*)^* \cong \V$ as parabolic bundles.
\item[(iii)] For any cover $f : Y \to X$ we have $f^*(\V^*) \cong (f^*\V)^*$ as parabolic bundles.
\end{itemize}
\end{prop}

\begin{proof}
All the proofs are just routine checks and follows from the corresponding results for orbifold bundles.
\end{proof}

\end{document}